\documentclass[11pt]{article}

\usepackage[margin=1in]{geometry}

\usepackage{amsfonts}
\usepackage{amsmath}
\usepackage{amsthm} 
\usepackage{amssymb}
\usepackage{epsfig, graphics}
\usepackage{hyperref}
\usepackage{url}
\usepackage{color}
\usepackage{setspace}
\usepackage{float}
\usepackage{subfig}
\usepackage{colortbl}
\usepackage{bm}
\usepackage{bbm}
\usepackage{mathtools}
\usepackage{dsfont}
\usepackage{mathrsfs}

\usepackage[square, numbers, comma, sort&compress]{natbib}

\newcommand{\R}{\mathbb{R}}

\newcommand{\cB}{\mathcal{B}}
\newcommand{\cD}{\mathcal{D}}

\newcommand{\Puni}{ \cP_{\text{\tiny Unif}}([0,1]\times\cD)}

\newcommand{\ion}{\frac{i}{n}}
\newcommand{\jon}{\frac{j}{n}}
\newcommand{\con}{\frac{C}{n}}
\newcommand{\Ion}{\frac{1}{n}}
\newcommand{\kon}{\frac{k}{n}}
\newcommand{\tX}{\widetilde{X}}

\newcommand{\cS}{\mathcal{S}}
\newcommand{\RR}{\mathbb{R}}
\newcommand{\PP}{\mathbb{P}}
\newcommand{\cW}{\mathcal{W}}
\newcommand{\cM}{\mathcal{M}}
\newcommand{\EE}{\mathbb{E}}

\newcommand{\cI}{\mathcal{I}}

\newcommand{\cA}{\mathcal{A}}

\newcommand{\cL}{\mathcal{L}}
\newcommand{\cP}{\mathcal{P}}
\newcommand{\cC}{\mathcal{C}}

\newcommand{\FF}{\mathbb{F}}
\newcommand{\cF}{\mathcal{F}}
\newcommand{\bH}{\mathbb{H}}
\newcommand{\cJ}{\mathcal{J}}
\newcommand{\bJ}{\mathbb{J}}
\newcommand{\bS}{\mathbb{S}}

\newtheorem{theorem}{Theorem}[section]
\newtheorem{remark}[theorem]{Remark}

\newtheorem{example}[theorem]{Example}
\newtheorem{assumption}[theorem]{Assumption}
\newtheorem{definition}[theorem]{Definition}
\newtheorem{corollary}[theorem]{Corollary}
\newtheorem{proposition}[theorem]{Proposition}
\newtheorem{lemma}[theorem]{Lemma}

\def\red{\color{black}}

  \def\blu{\color{black}}

 \def\bluu{\color{black}}

\begin{document}

\title{Stochastic Graphon Games with Jumps and Approximate Nash Equilibria}

\author{
Hamed Amini \thanks{Department of Industrial and Systems Engineering, University of Florida, Gainesville, FL, USA, email:  aminil@ufl.edu} \and
 Zhongyuan Cao 
\thanks{INRIA Paris,  48 rue Barrault, CS 61534
75647 Paris Cedex, France,  and Universit\'e Paris-Dauphine, email: zhongyuan.cao@inria.fr }
\and
Agn\`es  Sulem
\thanks{INRIA  Paris,  48 rue Barrault, CS 61534
75647 Paris Cedex, France, email: agnes.sulem@inria.fr}}

\maketitle

\begin{abstract}
We study continuous stochastic games with heterogeneous mean field interactions and jumps on large networks and explore their limit counterparts. We introduce the graphon game model based on a controlled graphon mean field stochastic differential equation system with jumps, which can be regarded as the limiting case of a finite game dynamic system as the number of players tends to infinity. 
We examine the case of controlled dynamics, with control terms present in the drift, diffusion, and jump components.
We focus on the study of the limit theory and provide convergence results on the state trajectories and their laws, transitioning from finite game systems to graphon systems. We also study approximate equilibria for finite games on large networks, using the graphon equilibrium as a benchmark. The rates of convergence are analyzed under various underlying graphon models and regularity assumptions. 

 \bigskip


\end{abstract}

\section{Introduction}

The study of mean field systems with homogeneous interaction dates back to the works of Boltzmann, Vlasov, McKean, and others (see e.g., \cite{applebaum2011nonlinear, mckean1967propagation, kac1959probability}). The theory of mean field games (MFG), introduced by Lasry and Lions in~\cite{lasry2007mean} and Huang, Caines, and Malham\'e in \cite{huang2006large, huang2007large}, has attracted considerable attention and been extensively studied in recent decades; see, in particular, the book {\red \cite{carmonabook1}} and references therein. As both {\red large $n$ player games and limiting models are rather} tractable, the MFG theory has developed a diverse and broad range of applications. However, despite some MFG models incorporating heterogeneity in individual characteristics, the framework {\red of the}  MFG theory remains {\red  mainly}  confined to games with homogeneous interactions, when all players are symmetrically exchangeable.


The study of stochastic games on large networks presents significant challenges, as various $n$ player networks may yield different limits when $n$ {\red goes to} infinity, particularly in the context of games on sparse networks (see, e.g.  \cite{feng2020linear, case2021lacker}). Analyzing games on large networks or those with heterogeneous interactions often relies on a tractable limiting (continuum) model, which can, in turn, offer insights into large finite games.

Recently, the use of graphons has emerged as a model to analyse heterogeneous interaction in mean field systems and heterogeneous game theory, see in particular \cite{coppini2020weakly,caines2020graphon,caines2021graphon}. 
Graphons have been developed by Lov\'asz et al., see e.g.~\cite{lovasz2012, borgs2008convergent, borgs2012convergent},  as a natural continuum limit object for large dense graphs.
Essentially, a graphon is a symmetric measurable function $G: I^2  \rightarrow I$, with $I:=[0,1]$ indexing a continuum of possible positions for nodes in the graph and $G(u,v)$ representing the edge density between nodes placed at $u$ and $v$. 
We refer {\red  to recent} papers by Bayraktar et al. \cite{bayraktar2020graphon,wu2022, bayraktar2022stationarity} for developments in the theory of graphon systems of interacting diffusions, the corresponding graphon-based limit theory, and propagation of chaos. These results are also applicable to graphon games on the underlying networks. Graphon static games have been studied in \cite{parise2023graphon, carmona2022stochastic}. For dynamic games, we refer to \cite{cui2021graphon} for discrete time models and \cite{gao2020graphon,graphon2022lauriere,lacker2021soret} for continuous time models. {\bluu In continuous-time graphon games, the linear-quadratic case has received particular attention, with approximate equilibria derived in \cite{gao2020graphon} and extended in \cite{graphon2022lauriere,lacker2021soret,hu2024finite}. For discrete-state models, \cite{aurell2022optimal} studies finite-state graphon games with jump dynamics motivated by epidemic control. Recent advances include nonlinear graphon limits \cite{coppini2024nonlinear} and optimal control in mean-field systems with heterogeneous and asymmetric interactions \cite{de2024mean,cao2025probabilistic}.} Our paper is closely related to \cite{lacker2021soret}, which uses the concept of graphon equilibrium to construct approximate equilibria for large finite games on any weighted, directed graph that converges in suitable norm. However, 
\cite{lacker2021soret}  does not consider direct interactions in the dynamics, and the heterogeneous interactions are only present in the reward function.

In this paper we study  graphon games with heterogeneous interactions and jumps and develop the associated limit theory. The traditional MFG framework is based on a fixed point problem describing the law of the state process $(X(t))_{t\in[0,T]}$ of a typical player. In the graphon
game model, we consider a fixed point problem for a family of laws $(X_u(.))_{u\in I}$. 
{\red We  include jumps} in the dynamics to model the instantaneous impacts. The jumps are induced by Poisson random measures with different compensator measures for different labels, which is a source of individual heterogeneity. 
Furthermore, the graphon interaction is present in the diffusion term, which is not the case  in the models in \cite{graphon2022lauriere, lacker2021soret}.
In addition, we incorporate a control variable, not only in the drift, as in \cite{graphon2022lauriere, lacker2021soret}, 
but also in the diffusion and jump terms.  More heterogeneity is thus introduced into our setup, and the interacting dynamic system is more complex compared to \cite{bayraktar2020graphon, wu2022}. The analysis becomes more involved in  the connection between finite games and graphon games. Our previous paper \cite{amicaosu2022graphonbsde} 
studies graphon mean field Backward Stochastic Differential Equations (BSDEs) with jumps and the associated limit theory. {\red Although obtained for backward systems,}
 some results on propagation of chaos are useful here for the analysis of our graphon games.

Working directly with a continuum of players, driven by a continuum of independent Brownian motions and independent Poisson random measures, raises  technical difficulties  since neither the map $I\ni u\mapsto W_u$ nor $I \ni u\mapsto N_u$ is  measurable. Noting that the value function is determined by the law of the  state processes $\cL(X_u), u\in I$,  we handle this issue by arguing that the laws $\cL(X_u)$ depend measurably on $u$, similarly to  \cite{bayraktar2020graphon,wu2022,amicaosu2022graphonbsde} but extended to a jump framework with controls. The different compensator measures of the jump processes also {\red increase the difficulty of handling the measurability issue}. We introduce a canonical coupling that allows us to obtain  measurability in a  stronger topology for the state processes $(X_u)$. Such a coupling has no influence on the graphon game, and leads to a straightforward way to investigate the connection between graphon games and finite games.




The graphon mean field model with jumps {\red may be  useful in various fields}, especially in finance. For instance, consider a financial network consisting of banks or investment firms with internal links and external investments. The internal links within the network, such as shared liabilities, credit exposures, or interbank lending, can be represented by the graphon interaction, while the external investments made by each entity introduce outside risks. These external risks can be influenced by various factors, such as market fluctuations or global events, and are modeled by the Poisson random measures. This way, the model captures the complex interactions and risk exposures present in real-world financial networks. 

The paper is structured as follows. In Section~\ref{GP}, we introduce the probabilistic setup,  notation, and background on graphons. Section~\ref{Sec3} presents graphon games with jumps and  associated graphon equilibria. 
{ \bluu Section~\ref{stabi} is devoted to  stability and continuity results on  graphon mean field game  equilibria and associated controlled systems.}
In Section~\ref{finteg}, we study large finite network games with heterogeneous interactions and their limiting characteristics as the interaction matrix converges to a given graphon. In Section~\ref{ANE}, we investigate the approximate Nash equilibria of finite games. {\bluu We conclude in Section~\ref{sec:conclusion}.}



\section{Graphons and probabilistic set-up}\label{GP}


Let  $T>0$ be a fixed time horizon.
Let $(\mathcal{S},|\cdot|)$ be a generic normed space. Denote by $\cC([0,T],\mathcal{S})$ the space of continuous functions from $[0,T]$ to $\mathcal{S}$ and denote by $\cD([0,T],\mathcal{S})$ the space of RCLL (right continuous with left limits) functions from $[0,T]$ to $\mathcal{S}$, both equipped with the supreme norm defined as $\sup_{t\in[0,T]}|x_t|$ for any element $x$. Let $\cC:= \cC([0,T],\RR)$ and $\cD:= \cD([0,T],\RR)$ for short. Denote by $\cM^+(\mathcal{S})$ the set of nonnegative Borel measures on $\mathcal{S}$. Denote by $\cP(\mathcal{S})$ the set of probability measures on  $\mathcal{S}$. For a random variable $X$, $\cL(X)$ denotes the law of $X$. Denote $\text{Unif}[0,1]$ the uniform measure on $[0,1]$ and  $ \cP_{\text{\tiny Unif}}([0,1]\times\mathcal{S})$ the set of Borel probability measures on $[0,1]\times \mathcal{S}$ with uniform first marginal. Denote by $ \cM^+_{\text{\tiny Unif}}([0,1]\times\mathcal{S})$ the set of nonnegative Borel measures on $[0,1]\times \mathcal{S}$ with uniform first marginal. We equip all spaces of measures with the topology of weak convergence. For a sequence $\{ X_n \}_{n \in \mathbb{N}}$ of real-valued random variables on a probability space $(\Omega,\cF,\PP)$, and a sequence of real numbers $\{a_n\}_{n\in \mathbb{N}}$, we write $X_n=O_p(a_n)$ if $\PP(|X_n|\leq C |a_n|)\to 1$ as $n\to \infty$ for some constant $C$, and write $a_n=o(1)$ if $a_n\to 0$ as $n\to \infty$.   
%
%
%
Let  $I=[0,1]$. A graphon is defined as a symmetric measurable function $G: I \times I \rightarrow I$. Graphons can be regarded as the limits of edge matrices of weighted {\bluu dense}  graphs, when the size of graphs (number of vertices) $n\in \mathbb{N}$ becomes 
large. 
We can think of $[0,1]$ as a continuum of indices for the vertices of a graph by $i/n$, $i\in [n]:=\{1,\dots, n\}$. Let $\mathcal{B}(I)$ be the Borel algebra on $I$. The so-called cut norm of a graphon is defined by
$\|G\|_{\square}:=\sup_{A,B\in\mathcal{B}(I)}\Bigl| \int_{A\times B}G(u,v)dudv\Bigr|.$
We can also view a graphon as an operator from $L^{\infty}(I)$ to $L^1(I)$, associating any {\bluu $\phi\in L^{\infty}(I)$} with 
$G\phi(u):=\int_I G (u,v) \phi(v) dv.$ 
By Lov\'asz \cite[Lemma 8.11]{lovasz2012}, the resulting operator norm turns out to be equivalent to the cut norm
$\| G\|_{\square}\leq  \|G \|_{\infty\to 1}\leq 4\| G\|_{\square},$ 
with 
$\|G \|_{\infty\to 1}:=\sup_{|\phi|\leq 1} \|G\phi\|_{1},$
where $\|\cdot\|_{1}$ denotes the $L^1$ norm. 
These norms are used in the study of convergence theorems for  graphon systems induced by a sequence of graphons.
 To study stronger convergence results, we need to consider another operator norm for graphons. We can view $G$ as an operator from $L^{\infty}(I)$ to $L^{\infty}(I)$ with the norm defined by
$\|G \|_{\infty\to \infty}:=\sup_{|\phi|\leq 1} \|G\phi\|_{L^\infty}.$  
We  
 restrict here to graphons  satisfying $\sup_{u\in I}\|G(u,\cdot)\|^{-1}_1<\infty$.
Let $\cP_2(\mathcal{S})$ be the subset of $\cP(\cS)$ with finite second moment, equipped with the Wasserstein-2 metric, where the Wasserstein-2 metric between two probability measures $\mu, \nu\in \cP_2(\cS)$ is defined as: 
\begin{align*}
\cW_2(\mu,\nu):=\Bigl(\inf\Bigl\{\EE\Bigl[|X_1-X_2|^2\Bigr]: \cL(X_1)=\mu,\cL(X_2)=\nu\Bigr\}\Bigr)^{1/2}. 
\end{align*}
For $\mu,\nu$ in $\cP(\cD([0,T],\cS))$, we define
\begin{align*}
\cW_{2,T}(\mu,\nu):=\Bigl(\inf\Bigl\{\sup_{t\in[0,T]}\EE\Bigl[|X_1(t)-X_2(t)|^2\Bigr]: \cL(X_1)=\mu,\cL(X_2)=\nu\Bigr\}\Bigr)^{1/2},  
\end{align*}
whenever it is well-defined.



Let $(\Omega,\cF,\mathbb{F},\PP)$ be a filtered probability space.
Let $I=[0,1]$ and 
$\{W_u:u\in I\}$ be a family of i.i.d. Brownian motions defined on $(\Omega,\cF,\PP)$. Let $\{N_u(dt, de):u\in I\}$ be a family of independent Poisson measures defined on $(\Omega,\cF,\PP)$ with compensator $\nu_u(de)dt$ such that $\nu_u$ is a measure on  $E:= \R_*$ with finite second moment, with $\R_*:= \R \setminus \{0\}$, equipped with its Borelian $\sigma$-algebra  ${\cal B} ( E)$, 
for each $u\in I$. Let $\{\widetilde{N}_u(dt, de):u\in I\}$ be their compensator processes. Let $\mathbb{F}=\{\cF_t,t\geq 0\}$ be the natural filtration associated with $\{W_u:u\in I\}$ and $\{N_u(dt, de):u\in I\}$ {\blu completed by all $\PP$-zero sets and augmented by the $\sigma$ field of initial condition $\cF_0$. We shall define the initial conditions of all label $u\in I$ on $\cF_0$.} {\blu Denote by $\mathbb{F}^u=\{\cF^u_t,t\geq 0\}$ the natural filtration of label $u$, generated by $W_u, N_u$, augmented by $\cF_0$, and completed by all $\PP$-zero sets.}
Let  $T>0$ be a fixed time horizon.
Denote by $P$ the predictable $\sigma$ algebra on $[0,T]\times \Omega$, {\blu and $P^u$ be the predictable $\sigma$ algebra of $\FF^u$}. \\
Below, we define   spaces of processes and random variables, with values in $\RR$.
 {\bluu Let $\mathbb{J}=\{\cJ_t,t\geq 0\}$ be a filtration. Then, for $t\in[0,T]$, 
\\ $\bullet$ $L^2(\cJ_t)$ is the set of all $\cJ_t$-measurable and square integrable random variables; 
\\ $\bullet$ $\mathbb{S}^2_T(\bJ)$ is the set of real-valued RCLL $\bJ$-adapted processes $\phi$ with \\ $\|\phi\|_{\mathbb{S}^2_T}:=\Bigl(\EE\Bigl[\sup_{t\in[0,T]}|\phi_t|^2\Bigr]\Bigr)^{1/2}<\infty$; 
\\ $\bullet$ $\mathbb{H}^2_T(\bJ)$ is the set of real-valued $\bJ$-adapted processes $\phi$ with \\$\|\phi\|_{\mathbb{H}^2_T}:=\Bigl(\EE\Bigl[\int_0^T |\phi_t|^2dt\Bigr]\Bigr)^{1/2}<\infty$; 
\\ $\bullet$ $\cM \mathbb{S}^2_T$ is the set of all families of processes $X=\{X_u\}_{u\in I}$ satisfying for each $u\in I$, $X_u\in \mathbb{S}^2_T(\FF^u)$, $I\ni u\mapsto \cL(X_u)$ is measurable and \\
$\sup_{u\in I}\|X_u\|^2_{\mathbb{S}^2_T}:=\sup_{u\in I}\EE\Bigl[\sup_{t\in[0,T]} |X_u(t)|^2\Bigr]<\infty.$ \\
For such $X\in \cM \mathbb{S}^2_T$, we define the norm
$\|X\|^{I}_{\mathbb{S}^2_T}:=\sup_{u\in I}(\EE\Bigl[\sup_{t\in[0,T]} |X_u(t)|^2\Bigr]\Bigr)^{1/2};$ \\
$\bullet$ $\cM \mathbb{H}^2_T$ is defined similarly, replacing $\mathbb{S}^2_T$ with $\mathbb{H}^2_T$; \\
$\bullet$ $\cM L^2_t$ is the set of all families of random variables $X(t)=\{X_u(t)\}_{u\in I}$ satisfying for each $u\in I$, $X_u(t)\in L^2(\cF^u_t)$, $I\ni u\mapsto \cL(X_u(t))$ is measurable and \\
$\|X_u(t)\|^I_{L^2}:=\sup_{u\in I}[\EE |X_u(t)|^2]^{1/2}<\infty.$\\
$\bullet$ $\cM\cP_2$ is the set of all families of continuous distributions $\{\mu_u\}_{u\in I}$ on $\RR$ satisfying $u\mapsto\mu_u$ is measurable and $\mu_u\in\cP_2(\RR)$ for all $u\in I$.	
}

If the underlying domain of related processes or random variables is not $\RR$ but  a {\bluu normed space $(\mathcal{S},\|\cdot\|_{\mathcal{S}})$, we define spaces similarly as above, replacing  the one-dimensional Euclid space $(\RR,|\cdot|)$ by $(\mathcal{S},\|\cdot\|_{\mathcal{S}})$, and then denote these by, for instance, $\bH^2_T(\mathcal{S},\FF^u)$, $\cM\bH^2_T(\mathcal{S})$ and $\cM\cP_2(\cS)$.}


Given a metric space $\mathcal{S}$, we define the measure-valued function $\Lambda\mu:[0,1]\to\cM^+ (\mathcal{S})$ for any $\mu\in  \cM\cP_2(\cS)$ as follows:
\begin{equation}\label{Lamb}
\Lambda\mu(u)(dx):=\int_{I}G(u,v)\mu(dv,dx).
\end{equation}

Similarly we define $\bar{\Lambda}\mu:[0,1]\to \cP_2(\mathcal{S})$ as 
	\begin{equation}\label{Lamb1}
		\bar{\Lambda}\mu(u)(dx):=\frac{1}{\|G(u,\cdot)\|_1}\int_{I}G(u,v)\mu(dv,dx). 
	\end{equation}

For any bounded measurable function $\phi: \mathcal{S}\to \RR$, we define the following
$$\langle \Lambda \mu(u),\phi \rangle :=\int_{[0,1]\times \mathcal{S}}G(u,v)\phi (x)\mu(dv,dx),$$
and the same for $\bar{\Lambda}\mu$.





\section{Graphon mean field games (GMFG) with jumps}\label{Sec3}
{\bluu
This section is dedicated to the main results on GMFG with jumps and associated  graphon equilibrium analysis. We restrict ourselves to the one-dimensional case, 
but the results can be generalized to a multi-dimensional set-up.
 Let the action space $(A,|\cdot|)$ be a Banach space. 
We consider control profiles of the following form: for each $u\in I$,
\begin{equation}\label{eq:controldef}
\alpha_u(t)=a_u(t,W_u(\cdot\wedge t),N_u(\cdot\wedge t),\xi_u), \quad t\in [0,T],
\end{equation}
where $\xi:=\{\xi_u\}_{u\in I}$ is the initial condition with  $\xi_u\in L^2(\cF_0)$ for each $u\in I$ and $\{\xi_u\}_{u\in I}$ are independent, $(a_u)_{u\in I}$ is a family of Borel measurable functions
$$
a_u:[0,T]\times \cC\times\cD\times \RR\to A.
$$
We also require that 
$$\sup_{u\in I}\int_0^T\EE[|\alpha_u(t)|^2]dtdu<\infty.$$
Note that for each $u\in I$, the process $\alpha_u$ of the above form is $\FF^u$-progressively measurable and $\alpha_u\in\bH^2_T(A,\FF^u)$. 
We denote by $\cA$ the set of all control profiles satisfying the above, and use $\cA^u$ to denote the set of controls in $\cA$ for label $u$. The set of admissible control profiles will be a subset of $\cA$. We will specify  below additional conditions to guarantee the measurability of the controlled system with jumps.  

}

The dynamics of the controlled graphon system is  as follows, 
\begin{align}
dX^{\alpha}_u(s)= & \int_{I}\int_{\RR}G(u,v)b(s,X^{\alpha}_u(s),x,\alpha_u(s))\mu^{\alpha}_{v,s}(dx)dvds  \label{eq:limitMFBSDE} \\
&+\int_{I}\int_{\RR}G(u,v)\sigma(s,X^{\alpha}_u(s),x,\alpha_u(s))\mu^{\alpha}_{v,s}(dx)dvdW_u(s) \nonumber\\
&+\int_{E}\ell(s^-,X^{\alpha}_u(s^-),e,\alpha_u(s^-))\widetilde{N}_u(ds,de),\quad  X_u(0)=\xi_u, \quad u\in I, \nonumber
\end{align}
where $\mu^{\alpha}_v:=\cL(X^{\alpha}_v)\in \cP(\cD)$ and $\mu^{\alpha}_{v,s}:=\cL(X^{\alpha}_v(s))\in \cP(\RR)$. \\
We make the following assumptions on the coefficients. 
{\bluu
\begin{assumption}\label{ass-coef}
\begin{itemize}
\item The measurable coefficients $b: [0,T]\times\RR\times \RR\times A\to \RR$, $\sigma: [0,T]\times\RR\times \RR\times A\to \RR$ and $\ell: [0,T]\times\RR\times E \times A\to \RR$ are Lipschitz continuous with respect to all parameters except $t$. \\ The functions $[0,T]\ni t \mapsto(b,\sigma)(t,0,0,0), \ell(t,0,e,0)$ are square integrable for all $(e)\in E$. 
\item For each $u\in I$, the initial law $\mu_{u,0} := \cL(\xi_u)$ satisfies, for all $ B\in\mathcal{B}(\RR^d)$, $I\ni u\mapsto \mu_{u,0}(B)$ is $I$-Lebesgue measurable and 
$\sup_{u\in I}\int_{\RR^d}|x|^{2}\mu_{u,0}(dx)<\infty.$
\item  There exists a coupling of the initial condition $\xi$ such that there exists $\bar{\xi} = \{\bar{\xi}_u\}_{u \in I}$ with $\cL(\bar{\xi}) = \cL(\xi)$ and $I\ni u\mapsto \bar{\xi}^u\in L^2(\RR^d)$ is measurable. 
\end{itemize}
\end{assumption}
}

To ensure measurability with respect to the  label $u$, we also make the following assumption on the intensity  of the Poisson measures:  
\begin{assumption}[Intensity measure]\label{assum:compensator}
  For each $u\in I$, let $\nu_u$ have continuous support. The function {\blu $I\times [1,2] \ni (u, w) \mapsto \varphi^{-1}_{u}(w-1) \in \RR$ is $\cB(I)\otimes\cB([1,2])$ measurable}, where $\varphi_u$ denotes the cumulative distribution function of $\nu_{u}$. We define {\bluu $\varphi_u^{-1}(1)$} as the essential supremum and {\bluu $\varphi_u^{-1}(0)$} as the essential infimum.
\end{assumption}

By  definition of $\Lambda\mu$, $\int_{\RR}G(u,v)b(s,X^{\alpha}_u(s),x,\alpha_u(s))\mu^{\alpha}_{v,s}(dx)dv$ can be expressed as  $\int_\RR b(t,X^{\alpha}_u(t),x,\alpha_u(t))\Lambda\mu^{\alpha}_t(u)(dx)$, that we simply write  $b(t,X^{\alpha}_u(t),\Lambda\mu_t(u),\alpha_u(t))$, \\
 viewing $\Lambda\mu_t(u)$ as a parameter.
We use the same notation rule for $\sigma$.\\
Note that in  the dynamics, the mean field interaction is linear, whereas in the objective function defined below, we allow for nonlinear dependence on the mean field term.
Given a fixed distribution  $\mu\in \cM\cP_2(\cD)$, 
{\bluu the objective of player $u$ is to choose a control $\alpha_u\in\cA^u$
to maximize the following expected payoff: 
\begin{equation}\label{GG}
J^u_G(\mu,\alpha_u):=\EE\Big[\int_0^T f(t,X^{\alpha}_u(t),\bar{\Lambda}\mu_t(u),\alpha_u(t))dt+g(X^{\alpha}_u(T),\bar{\Lambda}\mu_T(u))\Big].
\end{equation}
}
We make the following assumptions on the functions $f$ and $g$:

\begin{assumption}\label{assm:sta}
(i) The measurable functions $f: [0,T]\times \RR\times  \cP_2(\RR)\times A \to \RR $ and $g: \RR\times \cP_2(\RR) \to \RR$ are continuous w.r.t. all parameters except $t$.

		(ii) For some constant $L_q$, $f$ and $g$ satisfy the following:
		\begin{equation*}
			\begin{split}
				&\bigl\vert f(t,x',\mu',\alpha') - f(t,x,\mu,\alpha) \bigr\vert + \bigl\vert g(x',\mu') - g(x,\mu) \bigr\vert
				\\
				&\hspace{7pt} \leq L_q \bigl[ 1 + \vert x' \vert + \vert x \vert + \vert \alpha' \vert +
				\vert \alpha \vert + \|  \mu  \|_{2} + \|\mu' \|_{2} \bigr] \bigl[ \vert (x',\alpha') - (x,\alpha) \vert +
				W_{2}(\mu',\mu) \bigr].
			\end{split}
		\end{equation*} 	 
		
	(iii)  For each $(t,\mu,x)\in [0,T]\times \cP_2(\RR)\times \RR$, $A\ni \alpha\mapsto f(t,x,\mu,\alpha)$ is continuously differentiable and strictly concave with some constant $L_\lambda>0$ in the following sense,
	\begin{equation*}
		\begin{split}
			f(t,x,\mu,\alpha') - f(t,x,\mu,\alpha) 
			- \partial_{\alpha} f(t,x,\mu,\alpha)\cdot (\alpha'-\alpha) \leq -L_\lambda 
			\vert \alpha' - \alpha \vert^2. 
		\end{split}
	\end{equation*} 
	\item[(iv)]For any $t\in [0,T]$ and $\alpha\in A$, the map $\RR\times \cP_2(\RR)\ni (x,\mu)\mapsto \partial_\alpha f(t,x,\mu,\alpha)$ is Lipschitz continuous with a constant $L_f$, i.e.
	$$|\partial_\alpha f(t,x',\mu',\alpha)-\partial_\alpha f(t,x,\mu,\alpha)|\leq L_f (\cW_2(\mu',\mu)+|x-x'|).$$
\end{assumption}

\paragraph{Canonical coupling and measurability} Note that in the graphon game, the state dynamics of each label interact with other labels only through the laws of other labels. Thus when we couple the Brownian motions, the Poisson random measures and the control profiles in \eqref{eq:limitMFBSDE}, the joint law of the state and control $\cL(X_u,\alpha_u)$ for each label $u\in I$ does not change, and hence the value of the objective function remains the same. Therefore, we can study the dynamics through some coupling, under which the joint law of the trajectories of state and control for each label keeps the same and consequently the graphon equilibrium remains the same. In order to guarantee the well-posedness of the controlled graphon dynamics, we need the measurability of $u\mapsto \cL(X_u)$. If there is no jump included, we can simply take a common Brownian motion for all labels as in \cite[Lemma 2.1]{wu2022}, but  the presence of jumps here requires  additional care. This is achieved under Assumption \ref{assum:compensator}. 
Following  ideas of \cite{amicaosu2022graphonbsde} but for a forward graphon system, through a suitable coupling,which we call "canonical" coupling, 
we can obtain measurability of $I\ni u\mapsto \cL(X_u)$. 
Define the canonical filtered probability space $(\bar{\Omega},\bar{\cF},\bar{\mathbb{F}},\bar{\PP})$, where  $\bar{\mathbb{F}}=\{\bar{\cF}_t,t\geq 0\}$ is the completed natural filtration augmented by $\bar{\cF}_0$ and generated by a canonical one-dimensional Brownian motion $\bar{W}$ and a canonical Poisson random measure $\bar{N}(dt, de)$ with compensator $\nu(de)dt$, where $\nu$ is uniform on $[1,2]$. {\blu We denote by $\bar{\xi}_u$ the mirror of $\xi_u$ defined on $\bar{\cF}_0$, which admits the same distribution as $\xi_u$ for each $u\in I$, and satisfies that $u\mapsto \bar{\xi}_u$ is measurable in $L^2(\bar{\cF}_0)$. The existence of such $\{\bar{\xi}_u\}_{u\in I}$ is guaranteed by the third point of Assumption~\ref{ass-coef}.} 
Let $\cM\bS^2_T(\bar{\FF}), \cM\bH^2_T(\bar{\FF})$ be defined similarly as $\cM\bS^2_T, \cM\bH^2_T$ above except that for all $u\in I$, $X_u\in \bS^2_T(\bar{\FF})$ and $I\ni u \mapsto X_u$ is measurable in $\bS^2_T(\bar{\FF})$ (and similarly for $\cM\bH^2_T$).

{\bluu
	We define a family of Poisson random measures $\{\bar{N}_u\}_{u\in I}$ through a canonical Poisson random measure $\bar{N}$, as follows:
	For each $u\in I$,  
 every time  $\bar{N}$ has a jump of size $e$, the process $\bar{N}_u$ also jumps with a jump size $\varphi^{-1}_u(e-1)$. We then define
 \begin{equation}\label{eq:alphabar}
 \bar{\alpha}_u(t)\coloneqq a_u(t,\bar{W}_u(\cdot\wedge t),\bar{N}_u(\cdot\wedge t),\bar{\xi}_u), \quad t\in [0,T].
 \end{equation}
 In this way, for each control profile $\alpha\in\cA$, we couple it on $\bar{\Omega}$ with the identifier $\bar{\alpha}$ defined above. We now specify the set of \emph{admissible} control profiles.  
 The set $\cM\cA \subseteq  \cA$ consists of those control profiles $\alpha$ such that the associated identifier $\bar{\alpha}$ satisfies, for every $t \in [0,T]$, the mapping
  $$I\ni u\mapsto \bar{\alpha}_u(t)\in L^2(\bar{\Omega};A)$$
  is measurable. Note that if a control profile $\alpha$ is admissible, then necessarily $\alpha\in\cM\bH^2_T(A)$, and its identifier $\bar{\alpha}$ as defined in \eqref{eq:alphabar} belongs to $\cM\bH^2_T(A,\bar{\FF})$. 
} 

The canonically coupled controlled dynamics $\{X_u = X_u^{\bar{\alpha}}\}_{u\in I}$ is written as:
\begin{align}
& dX_u(s)=  \int_{\RR}b(s,X_u(s),x,\bar{\alpha}_u(s))\Lambda\mu_s(u)(dx)ds \label{eq:limitMFBSDE_couple}\\
&+\int_{\RR}\sigma(s,X_u(s),x,\bar{\alpha}_u(s))\Lambda\mu_s(u)(dx)d\bar{W}(s) \nonumber \\
&+\int_{E}\ell(s,X_u(s),{\bluu\varphi}^{-1}_u(e-1),\bar{\alpha}_u(s))\widetilde{\bar{N}}(ds,de),\quad  X_u(0)=\bar{\xi}_u, \;  u\in I,
\nonumber
\end{align}
where $\bar{\xi}$ is the coupled initial condition on $\bar{\Omega}$ satisfying the third point in Assumption~\ref{ass-coef}.

Using the above canonical coupling, we  obtain the following existence and uniqueness result for the controlled graphon dynamics.
{\bluu
\begin{theorem}[Well-posedness]\label{thm:exuni}
Let $\alpha \in \cM \bH^2_T(A)$ be an admissible control profile, and let $\bar{\alpha}$ be its one-to-one identifier in $\cM \bH^2_T(A,\bar{\FF})$. Then there exists a unique solution $\bar{X}$ to the coupled system \eqref{eq:limitMFBSDE_couple} such that $\bar{X} \in \cM \bS^2_T(\bar{\FF})$. Moreover, there exists a unique solution $X$ to the controlled graphon system \eqref{eq:limitMFBSDE} such that $X \in \cM \bS^2_T$.
\end{theorem}
}
\begin{proof}
Note that the solutions of \eqref{eq:limitMFBSDE_couple} and \eqref{eq:limitMFBSDE} admit the same distribution for each $u\in I$. Assume the first part of the theorem holds, we can plug the  distribution of the unique solution $\bar{X}$ of \eqref{eq:limitMFBSDE_couple} into the original system \eqref{eq:limitMFBSDE}, and 
 observe that the system has no interaction anymore. Then following standard arguments for proving the existence and uniqueness of SDEs with jumps (see e.g. \cite[Theorem V.32]{Protter}), it is clear that for each $u\in I$, there exists a unique solution $X_u$. Hence the system \eqref{eq:limitMFBSDE} admits a unique solution. In addition, the solution of \eqref{eq:limitMFBSDE_couple} is in $\cM\bS^2_T(\bar{\FF})$, which means $u\mapsto X_u$ is measurable in $\bS^2_T(\bar{\FF})$. When we decouple it to the original probability space $\Omega$, it will preserve the measurability in weak sense. Namely, the solution of \eqref{eq:limitMFBSDE} belongs to $\cM\bS^2_T$.   
Now let us prove the first part of Theorem \ref{thm:exuni}.  
For a fixed family of distributions $\{\mu_u\}_{u\in I}$ such that $u \mapsto\mu_u \in \cP(\cD)$ is measurable, let us first define
%
%
%
 the map $\mu\mapsto \Phi(\mu)$ by $\Phi(\mu):=(\cL(X^\mu_u):u\in I)$, where $X^\mu$ satisfies
 \eqref{eq:limitMFBSDE_couple} with fixed $\mu$.
  By the classical difference estimate of SDEs with jumps and a standard contraction argument, one can show that there exists a unique fixed point $\bar{\mu}\in\Puni$ such that $\bar{\mu}=\Phi(\bar{\mu})$. The pathwise uniqueness of the solution $X$ of \eqref{eq:limitMFBSDE_couple} also follows from the standard difference estimate assuming there exists two different solutions. We omit some details here, since the proof is similar to those existing in the graphon SDEs literature, despite the presence of the control processes. One can refer to \cite{bayraktar2020graphon,wu2022, amicaosu2022graphonbsde} for more details. 
  
  We now prove the measurability. First, we need to ensure the measurability of $X^\mu$ for each fixed measurable $\mu$, i.e., $X^\mu\in \cM\mathbb{S}^2_T(\bar{\FF})$.
By the preservation of measurability for the limit fixed point, we can obtain that the controlled state process $X$ belongs to the space $\cM\mathbb{S}^2_T(\bar{\FF})$. To do this,
 we define the  iterative equation
\begin{equation*}
\begin{split}
X^{(n)}_u(t)=&X^{(n-1)}_u(0)+  \int_0^t \int_{\RR}b(s,X^{(n-1)}_u(s),x,\bar{\alpha}_u(s))\Lambda\mu_s(u)(dx)ds \\
&+\int_0^t \int_{\RR}\sigma(s,X^{(n-1)}_u(s),x,\bar{\alpha}_u(s))\Lambda\mu_s(u)(dx)d\bar{W}(s)\\
&+\int_0^t \int_{E}\ell(s,X^{(n-1)}_u(s),\varphi^{-1}_u(e-1),\bar{\alpha}_u(s))\widetilde{\bar{N}}(ds,de), \quad u\in I,
\end{split}
\end{equation*}
with $X^0_u(t)\equiv X_u(0)$ for $t\in[0,T]$ and all $u\in I$. Now suppose $u\mapsto X^{(n-1)}_u$ is measurable. Then $u\mapsto\alpha(\cdot,u,\cdot)$ is also measurable by its definition. By the measurability of graphon $G(u,v)$, the map $u\mapsto \int_\RR b(s,x',x,a)\Lambda\mu_s(u)(dx)$ is measurable for any $(s,x',a)$. Hence,   $(u,s,x',a)\mapsto  \int_\RR b(s,x',x,a)\Lambda\mu_s(u)(dx)$ is measurable. Moreover since $b$ is Lipschitz continuous, we have that $(s,u,x',a)\mapsto b(s,x',x,a)$ is measurable. 
Now, we have that uniformly for $(s,x)\in[0,T]\times \RR$, $b(s,x',x,a)$ is continuous and grows at most linearly in $(x',a)$, and the same holds for  $\int_{\RR}b(s,x',x,a)\Lambda\mu_s(u)(dx)$ uniformly for $(s,u)\in[0,T]\times [0,1]$. It  follows by \cite[Lemma A.4]{wu2022} that 
$$I\ni u \mapsto \int_0^\cdot \int_{\RR}b(s,X^{(n-1)}_u(s),x,\bar{\alpha}_u(s))\Lambda\mu_s(u)(dx)ds\in \mathbb{S}^2_T(\bar{\FF})$$
is measurable. By similar arguments, we obtain measurability with respect to the volatility term and jump term, since they are now driven by a common Brownian motion and a common Poisson random measure. This implies that $X\in \cM\mathbb{S}^2_T(\bar{\FF})$ for the system \eqref{eq:limitMFBSDE_couple}.
\end{proof}

Let us now introduce the GMFG equilibria. 
{\bluu
\begin{definition}[GMFG Equilibrium]\label{def:graphonequi}
	A GMFG equilibrium is a distribution $\mu\in \cM\cP_2(\cD)$ such that there exists an admissible control profile $\alpha^\star\in \cM\cA$ satisfying
	\begin{equation*}
		J^u_{G}(\mu,\alpha^\star_u)=\sup_{\alpha\in  \cA^u}J^u_G(\mu,\alpha), \quad \mbox{for a.e.} \quad u\in I,
	\end{equation*}
	with $\mu_u=\cL(X^{\alpha^\star}_u)$ for all $u\in I$. Any $\alpha^\star$ satisfying the above is called an equilibrium control profile. 
\end{definition}

\begin{remark}
    For any $\alpha\in\cM\cA$, Theorem~\ref{thm:exuni} guarantees that $u\mapsto \cL(X^\alpha_u)$ is measurable. Hence $\{\cL(X^\alpha_u)\}_{u\in I}$ can be regarded as an element in $\cM\cP_2(\cD)$ and $\bar{\Lambda}\cL\big(X^\alpha(t)\big)(u)$ is well defined for any $t\in[0,T]$.
\end{remark}

The rest of this section discusses the existence and uniqueness of GMFG equilibrium, under an assumption that ensures a convex structure to the control problem and includes a monotonicity condition analogous to the classical Lasry–Lions condition: 
\begin{assumption}\label{assum-A2}
	\ 
	\begin{itemize}
    \item The action space $A$ is compact.
		\item 
		For each $(t,x,u,\mu)\in [0,T]\times \RR\times I \times \cM^+_{\rm Unif}(I\times \cD)$,  
		there exists $e\in E$ such that 
		the set 
		\begin{align*}
			K_e[\mu](t,x,u):=&
			\bigl\{\bigl(b(t,x,\Lambda\mu_t(u),a),\sigma^2(t,x,\Lambda\mu_t(u),a),\ell(t,x,e,a),z\bigr):\\
			& \qquad a\in A, z\leq f(t,x,\bar{\Lambda}\mu_t(u),a)\bigr\}
		\end{align*}
		is convex. 
		\item The map $e \mapsto \ell (t,x, e , a)$ is affine for each $(t,x,a) \in 
		[0,T]\times\RR\times  A$.
		
		\item For each $a\in A$, and any  $\mu_{1},\mu_{2}\in \cP_{\rm Unif}(I\times \RR\times A)$, we have 
		\begin{equation*}
			\begin{split}
				&\int_{[0,1]\times \RR\times A}\Bigl( f((t,x,\bar{\Lambda}\bar{\mu}_{1}(u),a)-f(t,x,\bar{\Lambda}\bar{\mu}_{2}(u),a) \Bigr)(\mu_{1}-\mu_{2})(du,dx,da) <0 ,\\
				& \text{and}\\
				&\int_{[0,1]\times \RR}\Bigl( g(x,\bar{\Lambda}\bar{\mu}_{1}(u))-g(x,\bar{\Lambda}\bar{\mu}_{2}(u)) \Bigr)(\bar{\mu}_1-\bar{\mu}_2)(du,dx) <0,
			\end{split}
		\end{equation*}
		where $\bar{\mu}$ is the marginal distribution of the first two coordinates.
	\end{itemize}
\end{assumption}  

In \cite{amicaosu2024weak}, following the approach of~\cite{lacker2015ex} and under Assumption~\ref{assum-A2}, we construct a unique strict Markovian GMFG solution to the graphon game with jumps in a relaxed formulation, where the associated control profile is characterized by a measurable function $\hat{\alpha}: I \times [0,T] \times \RR \to A$. We refer to \cite{amicaosu2024weak} for the definition of a strict Markovian GMFG solution and to \cite{lacker2015ex,lacker2021soret} for more on the relaxed formulation. In contrast, the present paper adopts a strong formulation for GMFG equilibria. In general, the strong and relaxed formulations are not equivalent.  However, under additional Lipschitz property, we have the following theorem.

\begin{theorem}\label{thm:uni}
	Suppose Assumption~\ref{assum-A2} holds. Then we have:\\ 
	(i) There exists a unique strict Markovian GMFG solution, and the associated equilibrium control is a function $\hat{\alpha}:I\times[0,T]\times\RR\to A$. \\
	(ii)  If $\hat{\alpha}$ is Lipschitz continuous
	in $x$, then there exists a unique GMFG equilibrium.
\end{theorem}

\begin{proof}
Part (i) follows from  
\cite[Theorem 3.5, 3.7]{amicaosu2024weak}.  
 For part (ii), if  $\hat{\alpha}$ is Lipschitz continuous in $x$, we can plug this function into the dynamics \eqref{eq:limitMFBSDE} and the strong well-posedness of the dynamics under $\hat{\alpha}$ is  guaranteed by Theorem~\ref{thm:exuni}. Moreover one can verify that such control profile is admissible. On the other hand, for any given mean field flow, since we maximize the objective function on a larger set in the relaxed formulation, it is clear that $\hat{\alpha}$ is also optimal in the strong formulation, i.e.
$$J^u_{G}(\cL(X^{\hat{\alpha}}),\hat{\alpha}_u)=\sup_{\alpha\in  \cA^u}J^u_G(\cL(X^{\hat{\alpha}}),\alpha)$$ 
for a.e. $u\in I$ with $\hat{\alpha}_u=\hat{\alpha}(u,\cdot,X^{\hat{\alpha}^u}_u(\cdot))$. Hence $\hat{\alpha}$ is the unique equilibrium control and the corresponding flow of controlled state distribution is the unique GMFG equilibrium.
\end{proof}


\begin{remark}
	In the literature on MFGs under strong formulation, equilibrium controls are generally obtained as unique maximizers of the associated Hamiltonian. The Lipschitz continuity of these maximizers is a typical condition to ensure the strong well-posedness of the controlled dynamics. In classical MFG settings without jumps, similar strongly convex conditions as in Assumption~\ref{assm:sta} for the cost functions $f,g$, together with some mild regularity assumptions are often sufficient to guarantee the Lipschitz property, see e.g., \cite{delarue13proana}. In our context, providing sufficient conditions to ensure Lipschitz property 
	is a more challenging task, due to the presence of heterogeneous interactions and jumps.  
	\end{remark}  
     
The main objective of this paper is to develop the limit theory for graphon games with jumps and the approximate Nash equilibria for corresponding $n$-player games. In what follows, we do not restrict ourselves to Markovian or Lipschitz controls, but instead consider a broader class of admissible equilibrium controls. Throughout the analysis, we assume that the GMFG admits a unique equilibrium.
}

\section{Stability and continuity results}\label{stabi} 
In this section, we analyze the sensitivity of equilibrium states and controls to changes in graphon structure and labels.
We begin with a useful lemma. 	
\begin{lemma}\label{lem:G}
	For any two families of distributions  $\{\mu^u_1,\mu^u_2\}_{u\in I}\in \cM\cP_2$ and any two graphons $G_1,G_2$, there exists a constant $C$ such that for any $u\in I$, 
	$$\cW^2_2(\bar{\Lambda}_1\mu_1(u),\bar{\Lambda}_2\mu_2(u)) \leq C\bigg(\|G_1(u,\cdot)-G_2(u,\cdot)\|_1+\int_I \cW_2^2(\mu^u_1,\mu^u_2)du\bigg).$$
\end{lemma}	
\begin{proof}
{\bluu Fix a label $u \in I$. Define the subsets
$$I^+ := \left\{ v \in I : \frac{G_1(u,v)}{\|G_1(u,\cdot)\|_1} - \frac{G_2(u,v)}{\|G_2(u,\cdot)\|_1} > 0 \right\}, \quad I^- := I \setminus I^+. $$
We construct a coupling measure $\nu$ on $I^+ \times I^-$ such that for all $v \in I^+$,  
$$\nu(dv,I^-)=\frac{G_1(u,v)}{\|G_1(u,\cdot)\|_1}-\frac{G_2(u,v)}{\|G_2(u,\cdot)\|_1},$$
and for all $v\in I^-$,
$$ \nu(I^+,dv)=\bigg|\frac{G_1(u,v)}{\|G_1(u,\cdot)\|_1}-\frac{G_2(u,v)}{\|G_2(u,\cdot)\|_1}\bigg|.$$ 
Then by the definition of Wasserstein distance,
$$\cW^2_2(\bar{\Lambda}_1\mu_1(u),\bar{\Lambda}_2\mu_1(u))\leq\int_{I^+\times I^-}\cW^2_2(\mu_1^{u},\mu_1^v)\nu(du,dv).$$
}
Using the triangle inequality for $\cW_2$, for any $(u,v)\in I^+\times I^-$, we obtain
$$\cW^2_2(\mu_1^u,\mu_1^v)\leq \|\mu^u_1\|^2_2+\|\mu^v_1\|^2_2,$$
which implies
$$\cW^2_2(\bar{\Lambda}_1\mu_1(u),\bar{\Lambda}_2\mu_1(u))\leq \int_I\bigg|\frac{G_1(u,v)}{\|G_1(u,\cdot)\|_1}-\frac{G_2(u,v)}{\|G_2(u,\cdot)\|_1}\bigg|\|\mu^v_1\|_2^2 dv. $$
On the other hand, for some constant $C>0$, we have
$$\bigg|\frac{G_1(u,v)}{\|G_1(u,\cdot)\|_1}-\frac{G_2(u,v)}{\|G_2(u,\cdot)\|_1}\bigg|\leq C(|G_1(u,v)-G_2(u,v)|+\|G_1(u,\cdot)-G_2(u,\cdot)\|_1).$$
Finally, using again the triangle inequality, we get
$$\cW^2_2(\bar{\Lambda}_1\mu_1(u),\bar{\Lambda}_2\mu_2(u)) \leq \cW^2_2(\bar{\Lambda}_1\mu_1(u),\bar{\Lambda}_2\mu_1(u))+\cW^2_2(\bar{\Lambda}_2\mu_1(u),\bar{\Lambda}_2\mu_2(u)),$$
and the claim follows.
\end{proof}	

{\bluu 

To establish the relation between equilibrium controls and states, we set the following assumption which will be maintained throughout the rest of the paper. 
\begin{assumption}\label{assm:tech}
Let  $(\mu_t\in\cP_2(\RR))_{t\in[0,T]}$  be a given mean field flow, and denote by $X^\mu$ and $\alpha^\mu$ the optimal controlled state process and the corresponding control, respectively. We assume that for any two distinct flows $\mu^1, \mu^2$, the following holds:
$$\EE\bigg[\int_0^T\big(\partial_\alpha f(t, X^{\mu^1}_{t}, \mu^1_{t}, \alpha^{\mu^1}_{t}) 
- \partial_{\alpha} f(t, X^{\mu^2}_{t}, \mu^2_{t}, \alpha^{\mu^2}_{t})\bigr)
\cdot (\alpha^{\mu^1}_{t} - \alpha^{\mu^2}_{t})dt\bigg]\leq 0.$$ 
\end{assumption} 

The following result measures the sensitivity of equilibrium controls with respect to changes in the graphon structure.

\begin{proposition}[Stability of equilibrium controls]\label{lem:sta}
	For two graphons $G_1,G_2$, let $\alpha_1, \alpha_2$ and $X_1, X_2$ denote the associated equilibrium control profiles and state processes under their respective GMFG equilibria. Then, for some constant $C > 0$, the following estimate holds for all $u \in I$:
	\begin{align*}
		\int_0^T \EE |\alpha_{1,u}(t)-\alpha_{2,u}(t)|^2dt\leq& C \bigg(\|G_1(u,\cdot)-G_2(u,\cdot)\|_1\\
		&+\int_I\int_0^T\EE|X_{1,u}(t)-X_{2,u}(t)|^2dtdu\bigg).
	\end{align*}
\end{proposition}
}
\begin{proof}
By the strict concavity of $f$ in Assumption \ref{assm:sta}, we have that for each fixed $(t,\mu,x)\in [0,T]\times \cP(\RR)\times \RR$, and any $\alpha,\alpha'\in A$,
$$(\partial_\alpha f(t,x,\mu,\alpha') 
- \partial_{\alpha} f(t,x,\mu,\alpha))\cdot (\alpha'-\alpha) \leq -L_\lambda 
\vert \alpha' - \alpha \vert^2. $$
Then,  integrating w.r.t. $t$ over $[0,T]$, we get
$$\int_0^T (\partial_\alpha f(t,x_t,\mu_t,\alpha'_t) 
- \partial_{\alpha} f(t,x_t,\mu_t,\alpha_t))\cdot (\alpha'_t-\alpha_t)dt \leq -L_\lambda\int_0^T 
\vert \alpha'_t - \alpha_t \vert^2 dt, $$
for each deterministic process $(x_t)_{t\in [0,T]}$ and probability measure flow $(\mu_t)_{t\in [0,T]}$ and any deterministic control processes $(\alpha'_t,\alpha_t)_{t\in[0,T]}$.
Using Cauchy-Schwarz inequality, we obtain
$$\|\partial_\alpha f(\cdot,x_\cdot,\mu_\cdot,\alpha'_\cdot) 
- \partial_{\alpha} f(\cdot,x_\cdot,\mu_\cdot,\alpha_\cdot)\|_{L^2([0,T])}\cdot \|\alpha'_\cdot-\alpha_\cdot\|_{L^2([0,T])} \geq L_\lambda 
\|\alpha' - \alpha \|^2_{L^2([0,T])}. $$
Recall the definition of $\alpha_1,\alpha_2$ and $X_1,X_2$, and let $\mu_1,\mu_2$ be the mean field associated to $X_1,X_2$ respectively. Here for notation convenience, we omit the script of label $u$ and notice that each triplet $(X_1,\mu_1,\alpha_1)$ belongs to the same label. 


By  Assumption~\ref{assm:tech}, we have
$$\EE\bigg[\int_0^T \big(\partial_\alpha f(t, X_{1,t}, \mu_{1,t}, \alpha_{2,t}) 
- \partial_{\alpha} f(t, X_{2,t}, \mu_{2,t}, \alpha_{1,t})\bigr)
\cdot (\alpha_{1,t} - \alpha_{2,t})dt\bigg]\leq 0.$$
Using the above two results and the strict concavity, we get
\begin{align*}
    &\int_0^T \EE\bigl[(\partial_\alpha f(t, X_{1,t}, \mu_{1,t}, \alpha_{1,t}) 
    - \partial_{\alpha} f(t, X_{2,t}, \mu_{2,t}, \alpha_{1,t})) 
    \cdot (\alpha_{1,t} - \alpha_{2,t})\bigr] \, dt \\
    &= \int_0^T \EE\bigl[
    \bigl(\partial_\alpha f(t, X_{1,t}, \mu_{1,t}, \alpha_{1,t}) 
    - \partial_\alpha f(t, X_{1,t}, \mu_{1,t}, \alpha_{2,t})\cdot (\alpha_{1,t} - \alpha_{2,t}) \\
    &\qquad + \partial_\alpha f(t, X_{1,t}, \mu_{1,t}, \alpha_{2,t}) 
    - \partial_{\alpha} f(t, X_{2,t}, \mu_{2,t}, \alpha_{1,t})\bigr)
    \cdot (\alpha_{1,t} - \alpha_{2,t})\bigr] \, dt \\
    &\leq \int_0^T \EE\bigl[(\partial_\alpha f(t, X_{1,t}, \mu_{1,t}, \alpha_{1,t}) 
    - \partial_\alpha f(t, X_{1,t}, \mu_{1,t}, \alpha_{2,t}))
    \cdot (\alpha_{1,t} - \alpha_{2,t})\bigr] \, dt \\
    &\leq -L_\lambda \int_0^T \EE|\alpha_{2,t} - \alpha_{1,t}|^2 \, dt.
\end{align*}
It follows that for any $u\in I$,
\begin{align*}
    \int_0^T \EE|\alpha_{1,u}(t) - \alpha_{2,u}(t)|^2 \, dt 
    &\leq \frac{1}{(L_\lambda)^2} \int_0^T \EE\bigl| 
    \partial_\alpha f(t, X_{1,u}(t), \bar{\Lambda}_1 \mu_{1,t}(u), \alpha_{1,u}(t)) \\
    &\qquad - \partial_{\alpha} f(t, X_{2,u}(t), \bar{\Lambda}_2 \mu_{2,t}(u), \alpha_{1,u}(t)) \bigr|^2 \, dt.
\end{align*}
By Assumption \ref{assm:sta} (iv), the integral \(\int_0^T \EE|\alpha_{1,u}(t) - \alpha_{2,u}(t)|^2 dt\) is bounded from above by
$$\frac{L_f}{(L_\lambda)^2}\int_0^T(\EE|X_{1,u}(t)-X_{2,u}(t)|^2+\cW_2^2(\bar{\Lambda}_1\mu_{1,t}(u),\bar{\Lambda}_2\mu_{2,t}(u))) dt.$$
Finally, combining with Lemma \ref{lem:G}, we can conclude. 
\end{proof}

 As a consequence, we obtain  the following corollary stating the difference of equilibrium controls associated to different labels in a same graphon system. 
\begin{corollary}[Label difference of equilibrium controls]\label{coro:sta}
	For a given graphon $G$ and any labels $u_1,u_2\in I$, let $\alpha_{u_1},\alpha_{u_2}$ and $X_{u_1},X_{u_2}$ be the equilibrium controls and the state processes for labels $u_1$ and $u_2$, respectively. Then, for some constant $C > 0$, 
	$$\int_0^T \EE |\alpha_{u_1}(t)-\alpha_{u_2}(t)|^2dt\leq C \bigg(\|G(u_1,\cdot)-G(u_2,\cdot)\|_1+\int_0^T\EE|X_{u_1}(t)-X_{u_2}(t)|^2dt\bigg).$$
\end{corollary}
\begin{proof}
{\bluu
	This follows by applying a similar  argument as in the proof of Proposition~\ref{lem:sta}, noting that in the final step, the difference \(\int_0^T \EE|\alpha_{u_1}(t) - \alpha_{u_2}(t)|^2 dt\) is bounded from above by
	$\frac{L_f}{(L_\lambda)^2}\int_0^T(\EE|X_{u_1}(t)-X_{u_2}(t)|^2+\cW_2^2(\bar{\Lambda}\mu_{t}(u_1),\bar{\Lambda}\mu_{t}(u_2))) dt.$
}
\end{proof}

We now provide a result which measures the distance between the state processes induced by different graphons.

\begin{theorem}[Stability of graphon] \label{thm:sta}
 Let $X$ and $X^{(n)}$ be the solutions of \eqref{eq:limitMFBSDE} under their respective equilibrium control profiles $\alpha$ and $\alpha^{(n)}$, associated to graphons $G$ and $G_n$, and initial conditions $\xi$ and $\xi^{(n)}$. Then, for some constant $C>0$,
\begin{align*}
  \int_I \EE\Big[\sup_{t\in[0,T]}|X^{(n)}_u(t)-X_u(t)|^2  \Big]du \leq C \Bigl(\int_I \EE |\xi_{u}-\xi^{(n)}_{u}|^2du +\|G-G_n\|_{1}\Bigr).
\end{align*}
Moreover, 
we have
\begin{align*}
\sup_{u\in I}\|X^{(n)}_{u}(t)-X_{u}(t)\|^2_{\mathbb{S}^2_T} \leq C\Bigl(\sup_{u\in I} \EE |\xi^{(n)}_{u}-\xi_{u}|^2 +\sup_{u\in I}\|G(u,\cdot)-G_n(u,\cdot)\|_{1} \Bigr). 
\end{align*}
\end{theorem}
\begin{proof}
We use similar techniques as in~\cite{bayraktar2020graphon,amicaosu2022graphonbsde}. 
By the Burkholder–Davis–Gundy inequality, we have:
\begin{eqnarray}\label{eq13}
 \| X^{(n)}_u-X_u\|^2_{\mathbb{S}^2_T} 
  \leq C \int_0^T\EE \Bigl| \int_I \int_\RR G_n(u,v) b(s,X^{(n)}_u(s),x,\alpha^{(n)}_u(s))\mu^{(n)}_{v,s}(dx)dv \Bigr. \\
 -\Bigl. \int_I \int_\RR G(u,v) b(s,X_u(s),x,\alpha_u(s))\mu_{v,s}(dx)dv \Bigr|^2ds \nonumber \\
 + C \int_0^T\EE \Bigl| \int_I \int_\RR G_n(u,v) \sigma(s,X^{(n)}_u(s),x,\alpha^{(n)}_u(s))\mu^{(n)}_{v,s}(dx)dv \Bigr. \nonumber \\
 -\Bigl.  \int_I \int_\RR G(u,v) \sigma(s,X_u(s),x,\alpha_u(s))\mu_{v,s}(dx)dv \Bigr|^2ds \nonumber \\
  + C \EE\int_0 ^T\int_E \Bigl| \ell(s,X^{(n)}_u(s),e,\alpha^{(n)}_u(s))-\ell(s,X_u(s),e,\alpha_u(s)) \Bigr|^2 N_u(ds,de) 
+ C\EE \| \xi^{(n)}_u-\xi_u \|^2 . \nonumber
\end{eqnarray}
We compute the first term; by adding and subtracting terms, we obtain: 
\begin{eqnarray} \label{eq:stabi_1}
 \int_0^T\EE \Bigl| \int_I \int_\RR G_n(u,v) b(s,X^{(n)}_u(s),x,\alpha^{(n)}_u(s))\mu^{(n)}_{v,s}(dx)dv \nonumber \\
 \qquad \qquad \qquad \Bigl. -\int_I \int_\RR G(u,v) b(s,X_u(s),x,\alpha_u(s))\mu_{v,s}(dx)dv \Bigr|^2ds \nonumber  \\
 \leq  C\int_0^T\EE\Bigl[\bigl(\int_I \int_{\RR}b(s,X_{u}(s),x,\alpha_u(s)) (G(u,v)-G_n(u,v))\mu_{v,s}(dx)dv\bigr)^2\Bigr]ds \nonumber  \\
 \quad + C\int_0^T\EE\Bigl[\int_I\int_{\RR}\Bigl|b(s,X_{u}(s),x,\alpha_u(s))-b(s,X^{(n)}_{u}(s),x,\alpha^{(n)}_u(s))\Bigr|^2G^2_n(u,v)\mu_{v,s}(dx)dv\Bigr]ds \nonumber  \\
 \quad +  C\int_0^T\EE\Bigl[\int_I\Bigl|\int_{\RR}b(s,X^{(n)}_{u}(s),x,\alpha^{(n)}_u(s))G_n(u,v)[\mu_{v,s}-\mu^{(n)}_{v,s}](dx)\Bigr|^2dv\Bigr]ds . 
\end{eqnarray}
Denote the three terms on the right-hand side of inequality \eqref{eq:stabi_1} as $\cI^{(n),1}_u$, $\cI^{(n),2}_u$, and $\cI^{(n),3}_u$ respectively. {\bluu By Assumption~\ref{ass-coef}, the fact that $X$ are in the space $\cM\bS^2_T$ and the action space $A$ is compact, we have for any $u\in I$,
	\begin{align*} 
	&\int_0^T\EE\Bigl[ \int_{\RR}|b(s,X_{u}(s),x,\alpha_u(s))|^2 \mu_{v,s}(dx)\Bigr]ds\\
	\leq& C\int_0^T(\EE|X_u(s)|^2+\EE|X_v(s)|^2 +\EE|\alpha_u(s)|^2+|b(s,0,0,0)|^2)ds\leq C. 
	\end{align*}
}
Then by Cauchy-Schwarz inequality and the boundedness of graphons, we get that
$$\int_I \cI^{(n),1}_{u}du\leq \int_I\int_I |G(u,v)-G_n(u,v)|^2dvdu\leq C\|G_n-G\|_1.$$
By Corollary \ref{coro:sta}, we obtain 
\begin{align*}
	\int_I \cI^{(n),2}_udu &\leq C\int_0^T \int_ I \EE\Bigl[|X_u(s)-X^{(n)}_u(s)|^2+|\alpha_u(s)-\alpha^{(n)}_u(s)|^2 \Bigr] du ds\\
	&\leq  C (\int_I\int_I|G_n(u,v)-G(u,v)|dvdu+\int_I\int_0^T\EE|X_{u}(s)-X_{u}^{(n)}(s)|^2dsdu)\\
	& \leq C(\|G_n-G\|_{1}+\int_I\int_0^T\EE|X_{u}(s)-X_{u}^{(n)}(s)|^2dsdu).
\end{align*}
Further by the definition of Wasserstein-2 metric, we have
 $$\int_I \cI^{(n),3}_{u}du \leq C\int_0^T \int_I (\cW_{2}(\mu_{s,v},\mu^{(n)}_{s,v}))^2 dv\leq C\int_I \int_0^T  \EE|X_v(s)-X^{(n)}_v(s)|^2 dsdv.$$ 
We address the second term of \eqref{eq13} in the same manner. Now, for the third term of \eqref{eq13}, by using the Lipschitz property of $\ell$, we have 
\begin{align}\label{eq:jumpdiff}
\EE\int_0 ^T\int_E \Bigl| \ell(s,X^{(n)}_u(s),e,\alpha^{(n)}_u(s))-\ell(s,X_u(s),e,\alpha_u(s)) \Bigr|^2 N_u(ds,de) \nonumber \\ \leq C \int_0^T \EE |X^{(n)}_u(s)-X_u(s)|^2 ds.
\end{align} 
By combining all the results above and integrating over $I$, we  obtain:
\begin{align*}
    \int_I \EE\Bigl[\sup_{t \in [0,T]} |X^{(n)}_{u}(t) - X_{u}(t)|^2\Bigr] \, du 
    &\leq C \Bigl[\int_I \EE |\xi^{(n)}_{u} - \xi_{u}|^2 \, du 
    + \|G_n - G\|_1 \\
    &\quad + \int_0^T \int_I \EE\Bigl[\sup_{t \in [0,s]} |X^{(n)}_{u}(t) - X_{u}(t)|^2\Bigr] \, du \, ds\Bigr].
\end{align*}
Applying Gronwall's Lemma gives
 \begin{align*}
\int_I\EE\Bigl[\sup_{t\in[0,T]}|X^{(n)}_{u}(t)-X_{u}(t)|^2\Bigr] du\leq C\Bigl[\int_I \EE |\xi^{(n)}_{u}-\xi_{u}|^2du +\|G-G_n\|_{1}\Bigr].
\end{align*}
Now taking the supremum over $I$ instead of integrating for $\cI^{(n),1}_u,\cI^{(n),2}_u,\cI^{(n),3}_u$ and  considering $\sup_{u\in I}\|X^{(n)}_{u}(t)-X_{u}(t)\|^2_{\mathbb{S}^2_T}$,
it follows by similar arguments as above that 
$$
\sup_{u\in I}\|X^{(n)}_{u}(t)-X_{u}(t)\|^2_{\mathbb{S}^2_T} \leq C\Bigl[\sup_{u\in I} \EE |\xi^{(n)}_{u}-\xi_{u}|^2 +\sup_{u\in I}\| G(u,\cdot)-G_n(u,\cdot) \|_{1} \Bigr],
$$
which concludes the proof.
\end{proof}

\begin{proposition}\label{prop:sta}
	{\bluu We use the same notation as in Theorem~\ref{thm:sta}, and consider the case when the solutions $X, X^{(n)}$ are driven by a common control profile $\alpha \in \cM\cA$. Then, if $\|G_n - G\|_{1} \to 0$ and $\int_I \EE |\xi_{u} - \xi^{(n)}_{u}|^2 du \to 0$, we have
	$$ \int_I\|X^{(n)}_u(t)-X_u(t)\|^2_{\mathbb{S}^2_T}du\longrightarrow 0.$$
	Furthermore, if $\sup_{u\in I}\|G_n(u,\cdot)-G(u,\cdot)\|_{1}\to 0$ and $\sup_{u\in I}\EE |\xi_{u}-\xi^{(n)}_{u}|^2du\to 0$, then
	$$\sup_{u\in I}\|X^{(n)}_{u}(t)-X_{u}(t)\|^2_{\mathbb{S}^2_T}\longrightarrow 0.$$
}
\end{proposition}
\begin{proof}
When the controls are identical, the problem reduces to stability analysis in a non-controlled setting. The result follows by adapting the arguments in \cite{bayraktar2020graphon} and the techniques for handling jumps from \cite{amicaosu2022graphonbsde}.
\end{proof}



We  introduce below  two continuity assumptions on the graphon, the initial condition, and the jump measure, under which we shall   obtain continuity results for the state processes.


\begin{assumption}\label{assum:conti}
There exists a finite collection of intervals $\{I_i : i=1,\ldots, n\}$ such that $I=\bigcup_{i}I_i$ and, for each $i\in\{1,\ldots,n\}$, we have:
\begin{itemize}
\item[(i)]  $u\rightarrow \cL(\xi_u)$ is continuous a.e. on $I_i$ w.r.t. the $\cW_2$ metric.
\item[(ii)] For each $j \in\{1,\ldots,n\}$,  $G(u,v)$ is continuous in $u$ and $v$ a.e. on $I_i\times I_j$.
\item[(iii)] The compensator measure $\nu_u$ is continuous in $u$ for the Wasserstein distance $\cW_2$ on each $I_i$.
\end{itemize}
\end{assumption}

\begin{assumption}\label{assum:lip}
There exists a finite collection of intervals $\{I_i : i=1,\ldots, n\}$ such that $I=\bigcup_{i}I_i$, and for some constant $C$, we have for all $u_1,u_2\in I_i$, $v_1,v_2 \in I_j$, and $i,j\in\{1,\ldots,n\},$  
\begin{align}
 \cW_2(\cL(\xi_{u_1}),\cL(\xi_{u_2}))\leq& C|u_1-u_2|,\label{eq:lip_ini}\\
 |G(u_1,v_1)-G(u_2,v_2)|\leq& C(|u_1-u_2|+|v_1-v_2|), \quad \text{and} 
 \\
\cW_2(\nu_{u_1},\nu_{u_2})\leq& C|u_1-u_2|\label{eq:lip_jump}.
\end{align}
\end{assumption}

{\bluu
The following example illustrates a simple class of graphons satisfying both Assumptions~\ref{assum:conti} and~\ref{assum:lip}.

\begin{example}[Piecewise constant graphon] \label{ex:SBM} We call a graphon \emph{piecewise constant} if there exists a collection of intervals $\{I_i, i=1,\ldots ,k\}$ for some $k\in\mathbb{N}$ such that $I=\bigcup_{i=1}^k I_i$ and for all $u_1,u_2\in I_i$, $v_1,v_2\in I_j$, and $i,j\in\{1,\ldots, k\}$, we have $G(u_1,v_1)=G(u_2,v_2)$. Such a graphon corresponds to the \emph{stochastic block model} and can be thought of as a model of multi-type mean field games. 
\end{example}

}
     
We  then obtain the following result on the difference between labels within the same system.
{\bluu
\begin{lemma}\label{lem:conti} 
	 Let $\{\alpha_u\}_{u\in I}$ be an equilibrium control profile.\\
(i)(Continuity) Suppose  Assumption~\ref{assum:conti} holds. Then the maps $u \mapsto \cL(X^\alpha_u)$ and $u \mapsto \cL(\alpha_u)$ are continuous on $I_i$ with respect to the $\cW_{2,T}$ distance.\\
(ii)(Lipschitz continuity)  Suppose Assumption~\ref{assum:lip} holds. Then the maps $u \mapsto \cL(X^\alpha_u)$ and $u \mapsto \cL(\alpha_u)$ are Lipschitz continuous on $I_i$ with respect to the $\cW_{2,T}$ distance.
\end{lemma}
}
\begin{proof}
    We use continuity arguments similar to those in  \cite{bayraktar2020graphon,amicaosu2022graphonbsde}, but with a different coupling method. Under the (Lipschitz) continuity assumption of the graphon $G$, by coupling $X_{u_1}$ and $X_{u_2}$ through a common Brownian motion $W$ and a two dimensional random Poisson measure $N$, allowing $N_{u_1}$ and $N_{u_2}$ to jump simultaneously with jump sizes determined by a joint distribution $\nu_{u_1,u_2}$. We have
\begin{align*}
& \| X_{u_1}-X_{u_2}\|^2_{\mathbb{S}^2_T}
 \leq C \int_0^T\EE \Bigl| \int_I \int_\RR G(u_1,v) b(s,X_{u_1}(s),x,\alpha_{u_1}(s))\mu_{v,s}(dx)dv \\
& \quad \quad -\int_I \int_\RR G(u_2,v) b(s,X_{u_2}(s),x,\alpha_{u_2}(s))\mu_{v,s}(dx)dv \Bigr|^2ds\\
& \quad + C \int_0^T\EE \Bigl| \int_I \int_\RR G(u_1,v) \sigma(s,X_{u_1}(s),x,\alpha_{u_1}(s))\mu_{v,s}(dx)dv \\
& \quad \quad -\int_I \int_\RR G(u_2,v) \sigma(s,X_{u_2}(s),x,\alpha_{u_2}(s))\mu_{v,s}(dx)dv \Bigr|^2ds\\
& \quad + C \EE\int_0 ^T\int_E \Bigl| \ell(s,X_{u_1}(s),e_1,\alpha_{u_1}(s))\\
&\quad \quad  \quad -\ell(s,X_{u_2}(s),e_2,\alpha_{u_2}(s)) \Bigr|^2 N(ds,d(e_1,e_2))
+ C\EE | \xi_{u_1}-\xi_{u_2} |^2,
\end{align*}
where $N(ds,d(e_1,e_2))$ has compensator $dt\nu_{u_1,u_2}(d(e_1,e_2))$ and $\nu_{u_1,u_2}$ represents the coupled measure of $\nu_{u_1}$ and $\nu_{u_2}$. We construct the measure $\nu_{u_1,u_2}$ in a way such that the infimum of $\EE_{\nu_{u_1,u_2}}| X_1-X_2 |^2$ is attained, with $\cL(X_1)=\nu_{u_1}$ and $\cL(X_2)=\nu_{u_2}$.
We can easily estimate the first two terms on the right-hand side by using a similar approach as in the proof of Theorem~\ref{thm:sta}, by using the estimate result of Corollary~\ref{coro:sta}. Denote by $\cI$ the sum of the first two terms in the right-hand side of the above equation, we have
$$\cI\leq C\int_0^T \EE|X_{u_1}(s)-X_{u_2}(s) |^2ds+ CT\int_I  |G(u_1,v)-G(u_2,v)|dv.$$
For the third term, we have 
\begin{align*}
& \EE\int_0 ^T\int_E \Bigl| \ell(s,X_{u_1}(s),e_1,\alpha_{u_1}(s))-\ell(s,X_{u_2}(s),e_2,\alpha_{u_2}(s)) \Bigr|^2 N(ds,d(e_1,e_2)) \\
& \leq C \int_0^T \EE |X_{u_1}(s)-X_{u_2}(s)|^2 ds +CT\int_I  |G(u_1,v)-G(u_2,v)|dv+C|u_1-u_2|.
\end{align*}
It follows by Gronwall lemma that
$$\EE\| X_{u_1}-X_{u_2}\|^2_{\mathbb{S}^2_T} \leq C\EE | \xi_{u_1}-\xi_{u_2} |^2+CT\int_I  |G(u_1,v)-G(u_2,v)|dv+{\bluu C|u_1-u_2|}.$$
Now, by taking the infimum over random variables $\xi_{u_1}$ and $\xi_{u_2}$ and combining with Corollary~\ref{coro:sta}, we can conclude point (i) and (ii) under the respective continuity conditions and Lipschitz conditions.
\end{proof}

\section{Large finite  network  games with heterogeneous interactions}\label{finteg}
In this section, we study large finite network games with heterogeneous interactions and analyze their limiting characteristics as the number of players $n$ goes to  infinity, with the interaction matrix converging to a given graphon.
\subsection{Finite games with jumps}
Let $n\in \mathbb{N}$ be   the network's size. Consider an  heterogeneous interacting particle system $ X^{(n)} = X^{(n), \alpha}$ with controlled dynamics
\begin{align}
dX^{(n)}_i(s)= & \frac{1}{n}\sum_{j=1}^{n}\zeta^{(n)}_{ij}b(s,X^{(n)}_i(s),X^{(n)}_j(s),{\blu \alpha^{(n)}_i(s)})ds \label{eq:finiteMFBSDE}\\
&+ \frac{1}{n}\sum_{j=1}^{n}\zeta^{(n)}_{ij}\sigma(s,X^{(n)}_i(s),X^{(n)}_j(s),{\blu \alpha^{(n)}_i(s)})dW_{\frac i n}(s) \nonumber \\
&+\int_{E}\ell(s,X^{(n)}_i(s),e,{\blu \alpha^{(n)}_i(s)})\widetilde{N}_{\frac i n}(ds,de),\quad \quad X^{(n)}_i(0)=\xi^{(n)}_i. \nonumber
\end{align}
{\bluu Here, $\zeta^{(n)}:=(\zeta^{(n)}_{ij})_{ij}$ is an $n\times n$ symmetric nonnegative interaction matrix representing the strength or probability of interaction between players $i$ and $j$, and the admissible control for player $i\in [n]$ is a stochastic process $\alpha^{(n)}_i$ in $\cA^{\frac{i}{n}}$, where $\cA^{\frac{i}{n}}$ is defined in the same way as $\cA^u$, with $u=i/n$. The controls are considered in a distributed sense, meaning that $\alpha^{(n)}_i$ is progressively measurable with respect to the filtration $\FF^u$. We denote by $\cM\cA_n$ the set of admissible control profiles for the $n$-player game. 
We assume that $\xi^{(n)}_i\in L^2(\cF_0)$ for all $i=1, \dots, n$ and that they are independent; moreover  we assume that the coefficients $b$, $\sigma$, and $\ell$ satisfy the first point of Assumption~\ref{ass-coef}.
}


 



For each player $i\in[n]$, we define the neighborhood empirical measure as 
\begin{equation}\label{nem}
M^{(n)}_i:= \frac{1}{n}\sum_{j=1}^n\zeta^{(n)}_{ij}\delta_{X^{(n)}_i}\in \cM^+(\cD), 
\end{equation}
and the neighborhood empirical measure at time $s$ as 
\begin{equation*}
M^{(n)}_i(s):= \frac{1}{n}\sum_{j=1}^n\zeta^{(n)}_{ij}\delta_{X^{(n)}_i(s)} \in \cM^+(\RR). 
\end{equation*}
Further, we define the normalized neighborhood empirical measure as 
\begin{equation*}
	\overline{M}^{(n)}_i:= \frac{1}{\kappa^{(n)}_i}\sum_{j=1}^n\zeta^{(n)}_{ij}\delta_{X^{(n)}_i} \in \cP(\cD), 
\end{equation*}
which is a probability measure , where $\kappa^{(n)}_i:=\sum_{j=1}^n\zeta^{(n)}_{ij}$ is the total connectivity of player $i$.
{\bluu
Given a control profile $\bm{\alpha}^{(n)}\coloneqq(\alpha_1^{(n)},\alpha_2^{(n)},\ldots,\alpha_n^{(n)})$, the objective function of player $i\in[n]$ is 
$$J_i(\bm{\alpha}^{(n)})\coloneqq\EE\Bigl[ \int_0^T f(t,X^{(n)}_i(t),\overline{M}^{(n)}_i(t),\alpha_i^{(n)}(t))dt+g(X^{(n)}_i(T),\overline{M}^{(n)}_i(T)) \Bigr],$$
where the functions 
$f$  
and $g$ satisfy Assumption~\ref{assm:sta}. 
A Nash equilibrium is defined as a control profile $\bm{\alpha}^{(n)}\in \cM\cA_n$ such that for each $i\in [n]$,
$$J_i(\bm{\alpha}^{(n)})=\sup_{\beta\in\cA^{\frac{i}{n}}} J_i(\alpha_1^{(n)},\ldots,\alpha_{i-1}^{(n)},\beta,\alpha_{i+1}^{(n)},\ldots,\alpha_n^{(n)}).$$
For a vector $\bm{\epsilon}^{(n)}=(\epsilon_1,\ldots,\epsilon_n)\in[0,\infty)^n$, we say that $\bm{\alpha}^{(n)}$ is an approximate $\epsilon$-Nash equilibrium if, for each $i\in[n]$,
$$J_i(\bm{\alpha}^{(n)})\geq\sup_{\beta\in\cA^{\frac{i}{n}}}J_i(\alpha_1^{(n)},\ldots,\alpha_{i-1}^{(n)},\beta,\alpha_{i+1}^{(n)},\ldots,\alpha_n^{(n)})-\epsilon_i.$$
}

\subsection{Propagation of chaos for controlled graphon system}

We now use the results of Section \ref{stabi} to obtain  convergence results from the finite controlled system to the limiting graphon controlled system. 
To this purpose, we introduce the following regularity condition on the interaction strengths $\zeta^{(n)}$ in the $n$-player system.
\begin{assumption}[Regularity interaction]\label{assum:regu}
We say that  $\zeta^{(n)}:=\{\zeta^{(n)}_{ij}\}_{i,j\in[n]}$ satisfies the regularity assumption with respect to a graphon $G$ if either of the following holds: 
 $\zeta^{(n)}_{ij}=G(\frac{i}{n},\frac{j}{n})$ or 
 $\zeta^{(n)}_{ij}={\rm Bernoulli}\big(G(\frac{i}{n},\frac{j}{n})\big)$ independently for all $1\leq i \leq j \leq n$ and independent of $\{\xi_u, W_u, N_u: u\in I\}$.
\end{assumption}
We recall the following definition:
\begin{definition}
    We call $\{G_n\}_{n\in \mathbb{N}}$ a sequence of step graphons if, for each $n\in \mathbb{N}$, $G_n$ is a graphon that satisfies $G_n(u,v)=G_n\Bigl(\frac{\lceil nu\rceil}{n},\frac{\lceil nv\rceil}{n}\Bigr)$ for all $(u,v)\in I\times I$. 
\end{definition}

{\bluu
\begin{theorem}[Large population convergence]\label{thm:convg_path}
 Let $X$ and $X^{(n)}$ be the respective solutions 
of the  graphon system \eqref{eq:limitMFBSDE} with graphon $G$, and the $n$-player system \eqref{eq:finiteMFBSDE} with interaction matrix $\zeta^{(n)}:=(\zeta^{(n)}_{ij})_{ij}$, with respective initial conditions $\xi$ and $\xi^{(n)}$, and respective   controls $\alpha$ and $\alpha^{(n)}$, 
where 
$\alpha$ is an equilibrium control profile of the  graphon game~\eqref{GG}, and   $\alpha^{(n)}:=(\alpha^{(n)}_i)_{i\in[n]}$, with 
$\alpha^{(n)}_i=\alpha_{i/n}$.  
Suppose Assumption~\ref{assum:conti} holds for the graphon system \eqref{eq:limitMFBSDE}, and that $\zeta^{(n)}$ satisfies the regularity Assumption~\ref{assum:regu} with respect to a graphon $G_n$, where $\{G_n\}_{n}$ is a sequence of step graphons such that $ \|G-G_n\|_{1}\to 0$. Then, if 
$ \frac{1}{n}\sum_{i=1}^n\EE |\xi^{(n)}_{i}-\xi_{\ion}|^2\to 0$, the following convergence result holds for the mean empirical neighborhood measure (defined in \eqref{nem}) as $n\to\infty$: 
\begin{equation}\label{eq:dde}
\Ion \sum_{i=1}^n M^{(n)}_i\longrightarrow \int_I \Lambda\mu(v)dv,
\end{equation}
in probability with $\mu= \cL(X)$. Moreover, we also have
$$
\frac{1}{n}\sum_{i=1}^n \| X^{(n)}_i-X_{\frac{i}{n}}\|^2_{\mathbb{S}^2_T}\to 0.$$

Furthermore, if $\|G-G_n\|_{\infty\to \infty}\to 0$ and $ \max_{i\in [n]}\EE |\xi^{(n)}_{i}-\xi_{\ion}|^2\to 0$, then for each $i\in [n]$ and any bounded Lipschitz continuous function $h: \cD\to \RR$, as $n\to \infty$, 
$$\EE\Bigl[ \langle h, M^{(n)}_i\rangle -\langle h, \Lambda\mu(\ion)\rangle\Bigr]^2\longrightarrow 0 \quad \text{and} \quad   \max_{i\in [n]} \| X^{(n)}_{i }-X_{\ion} \|^2_{\mathbb{S}^2_T}\to 0.$$

In addition, if Assumption~\ref{assum:lip} holds, then for some constant $C > 0$, we have:
\begin{align}
&\frac{1}{n}\sum_{i=1}^n \| X^{(n)}_i-X_{\frac{i}{n}}\|^2_{\mathbb{S}^2_T} \leq C\Bigl( \frac{1}{n}\sum_{i=1}^n\EE |\xi^{(n)}_{i}-\xi_{\ion}|^2 +\|G-G_n\|_{1}+ \frac{1}{n} \Bigr), \label{eq:convg_3} \\
 &\max_{i\in [n]} \| X^{(n)}_{i }-X_{\ion} \|^2_{\mathbb{S}^2_T}  \leq C\Bigl( \max_{i\in[n]} \EE |\xi^{(n)}_{i}-\xi_{\ion}|^2 +\max_{i\in [n]}\|G(\ion,\cdot)-G_n(\ion,\cdot)\|_{1}+ \frac{1}{n} \Bigr),\label{eq:convg_4} \\
 &\EE\Bigl[ \frac 1 n\sum_{i=1}^n(\langle h, M^{(n)}_i\rangle -\langle h,\Lambda\mu(\ion)\rangle)\Bigr]^2\leq C\Big(\frac{1}{n}\sum_{j=1}^n \EE| \xi^{(n)}_j-\xi_{\jon}|^2+ \| G_n-G\|_{1}+ \frac{1}{n}\Big).\label{eq:newadd}
 \end{align}
 Moreover, for each $i\in[n]$, 
\begin{equation}\label{eq:sss}
\EE\Bigl[ \langle h, M^{(n)}_i\rangle -\langle h, \Lambda\mu(\ion)\rangle\Bigr]^2\leq \frac{C}{n}\sum_{j=1}^n \EE| \xi^{(n)}_j-\xi_{\jon}|^2+C \| G_n-G\|_{\infty\to \infty} + \frac{C}{n}.
\end{equation}
\end{theorem}
}
\begin{proof} 
We first prove the upper bound of $\frac{1}{n}\sum_{i=1}^n \| X^{(n)}_i-X_{\frac{i}{n}}\|^2_{\mathbb{S}^2_T}$.
First, we estimate the difference between $X^{(n)}$ and $\tX^{(n)}$, where $\widetilde{X}^{(n)}$ is the solution of \eqref{eq:limitMFBSDE} under control $\alpha$, with graphon $G_n$ and initial condition $\widetilde{\xi}^{(n)}_u=\xi_u, u\in I$. By the Burkholder-Davis-Gundy inequality, we have (for some $C>0$): 
\begin{align*}
& \| X^{(n)}_i-\tX^{(n)}_{\frac{i}{n}}\|^2_{\mathbb{S}^2_T} \\
& \leq C \int_0^T\EE \Bigl| \frac{1}{n}\sum_{j=1}^n\zeta^{(n)}_{ij}b(s,X^{(n)}_i(s),X^{(n)}_j(s),\alpha^{(n)}_i(s))\\
&\qquad \qquad \qquad -\int_I \int_\RR G_n(\frac{i}{n},v) b(s,\tX^{(n)}_{\frac{i}{n}}(s),x,\alpha_{\frac{i}{n}}(s))\mu_{v,s}(dx)dv \Bigr|^2ds\\
& \quad + C \int_0^T\EE \Bigl| \frac{1}{n}\sum_{j=1}^n\zeta^{(n)}_{ij}\sigma(s,X^{(n)}_i(s),X^{(n)}_j(s),\alpha^{(n)}_i(s))\\
&\qquad \qquad \qquad-\int_I \int_\RR G_n(\frac{i}{n},v) \sigma(s,\tX^{(n)}_{\frac{i}{n}}(s),x,\alpha_{\frac{i}{n}}(s))\mu_{v,s}(dx)dv \Bigr|^2ds\\
& \quad + C \EE\int_0 ^T\int_E \Bigl| \ell(s,X^{(n)}_i(s),e,\alpha^{(n)}_i(s))-\ell(s,\tX^{(n)}_{\frac{i}{n}}(s),e,\alpha_{\frac{i}{n}}(s)) \Bigr|^2 N_{\frac{i}{n}}(ds,de)\\
& \quad +C \EE | \xi^{(n)}_i-\widetilde{\xi}^{(n)}_{\ion}|^2.
\end{align*}
Let us compute the difference in the first term 
of the right-hand side of the above equation, and proceed similarly for the second term. 
We have:
\begin{align*}
& \EE \Bigl| \frac{1}{n}\sum_{j=1}^n\zeta^{(n)}_{ij}b(s,X^{(n)}_i(s),X^{(n)}_j(s),\alpha^{(n)}_i(s)) \\
&\qquad \qquad \qquad -  \int_I \int_\RR G_n(\frac{i}{n},v) b(s,\tX^{(n)}_{\frac{i}{n}}(s),x,\alpha_{\frac{i}{n}}(s))\mu_{v,s}(dx)dv \Bigr|^2 \\
& \quad \leq 3(\cI^{(n),1}_s+\cI^{(n),2}_s+\cI^{(n),3}_s),
\end{align*}
with
\begin{align*}
\cI^{(n),1}_s :=& \EE \Bigl| \frac{1}{n}\sum_{j=1}^n\zeta^{(n)}_{ij}b(s,X^{(n)}_i(s),X^{(n)}_j(s),\alpha^{(n)}_i(s))\\
&\qquad \qquad \qquad -  \frac{1}{n}\sum_{j=1}^n\zeta^{(n)}_{ij}b(s,\widetilde{X}^{(n)}_{\ion}(s),\widetilde{X}^{(n)}_{\jon}(s),\alpha_{\frac{i}{n}}(s))\Bigr|^2,\\
\cI^{(n),2}_s :=&  \EE \Bigl| \frac{1}{n}\sum_{j=1}^n\zeta^{(n)}_{ij}b(s,\widetilde{X}^{(n)}_{\ion}(s),\widetilde{X}^{(n)}_{\jon}(s),\alpha_{\frac{i}{n}}(s))\\
&\qquad \qquad \qquad - \int_I \int_\RR G_n(\frac{i}{n},v) b(s,\widetilde{X}^{(n)}_{\frac{i}{n}}(s),x,\alpha_{\frac{i}{n}}(s))\widetilde{\mu}^{(n)}_{\frac{\lceil nv \rceil}{n},s}(dx)dv \Bigr|^2, \\
\cI^{(n),3}_s :=& \EE \Bigl| \int_I \int_\RR G_n(\frac{i}{n},v) b(s,\widetilde{X}^{(n)}_{\frac{i}{n}}(s),x,\alpha_{\frac{i}{n}}(s))\widetilde{\mu}^{(n)}_{\frac{\lceil nv \rceil}{n},s}(dx)dv \\
&\qquad \qquad \qquad - \int_I \int_\RR G_n(\frac{i}{n},v) b(s,\tX^{(n)}_{\frac{i}{n}}(s),x,\alpha_{\ion}(s))\widetilde{\mu}^{(n)}_{v,s}(dx)dv \Bigr|^2.
\end{align*}
Then since $\alpha^{(n)}_i=\alpha_{i/n}$, it follows by using the law of large numbers and  similar arguments as in the proof of \cite[Lemma 6.1]{bayraktar2020graphon}, that 
$$\frac{1}{n}\sum_{i=1}^n\cI^{(n),1}_s\leq C \frac{1}{n}\sum_{i=1}^n \EE\bigl| X^{(n)}_i(s)-\widetilde{X}_{\ion}(s)\bigr|^2,$$
and
$\cI^{(n),2}_s\leq \frac{C}{n}.$ 
In addition, we have by Proposition~\ref{lem:sta} and Lemma~\ref{lem:conti}, that $v\mapsto\widetilde{\mu}^{(n)}_{v,s}$ is Lipschitz continuous in Wasserstein-2 distance for any $s\in[0,T]$. Hence we have 
\begin{align*}
&\int_I \int_\RR G_n(\frac{i}{n},v) b(s,\widetilde{X}^{(n)}_{\frac{i}{n}}(s),x,\alpha_{\frac{i}{n}}(s))\widetilde{\mu}^{(n)}_{\frac{\lceil nv \rceil}{n},s}(dx)dv\\
&\qquad \qquad \qquad - \int_I \int_\RR G_n(\frac{i}{n},v) b(s,\widetilde{X}^{(n)}_{\frac{i}{n}}(s),x,\alpha_{\ion}(s))\widetilde{\mu}^{(n)}_{v,s}(dx)dv \\
& \quad \leq C\int_I \cW_2(\widetilde{\mu}^{(n)}_{\frac{\lceil nv \rceil}{n},s},\widetilde{\mu}^{(n)}_{v,s})dv \leq \con.
\end{align*}
For the jump term, we have similarly as in \eqref{eq:jumpdiff}
\begin{align*}
\frac{1}{n}&\sum_{i=1}^n\EE\int_0 ^T\int_E \Bigl| \ell(s,X^{(n)}_i(s),e,\alpha^{(n)}_i(s))-\ell(s,\tX_{\ion}(s),e,\alpha_{\ion}(s)) \Bigr|^2 N_u(ds,de)\\
&\qquad  \leq C \int_0^T \frac{1}{n}\sum_{i=1}^n\EE |X^{(n)}_i(s)-X_{\ion}(s)|^2 ds.
\end{align*}
Arguing similarly as in the proof of Theorem~\ref{thm:sta}, we get
\begin{equation} \label{eq00}
\frac{1}{n}\sum_{i=1}^n \| X^{(n)}_i-\tX^{(n)}_{\frac{i}{n}}\|^2_{\mathbb{S}^2_T} \leq \frac{1}{n}\sum_{i=1}^n\EE |\xi^{(n)}_{i}-\xi_{\ion}|^2+\frac{C}{n}.
\end{equation}
Noticing that  $ \frac{1}{n}\sum_{i=1}^n\EE |X^{(n)}_i(s)-\tX^{(n)}_{\ion}(s)|^2\leq \max_{i\in [n]}\EE |X^{(n)}_i(s)-\tX^{(n)}_{\ion}(s)|^2$, by repeating the above analysis and taking the maximum for $i\in[n]$ instead of the sum, we obtain
$$\max_{i\in [n]} \| X^{(n)}_i-\tX^{(n)}_{\frac{i}{n}}\|^2_{\mathbb{S}^2_T} \leq \max_{i\in [n]}\EE |\xi^{(n)}_{i}-\xi_{\ion}|^2+\frac{C}{n}.$$
Moreover, by Theorem~\ref{thm:sta}, we have 
\begin{equation}\label{eq01}
\int_I \| \tX^{(n)}_{u}-X_{u} \|^2_{\mathbb{S}^2_T} du \leq C\bigg(\int_I \EE |\widetilde{\xi}^{(n)}_{u}-\xi_{u}|^2du +\|G-G_n\|_{1}\bigg)=C\|G-G_n\|_1 .
\end{equation}
In addition, by following similar arguments as in the proof of Theorem~\ref{thm:sta}, we get
\begin{equation}\label{eq:99}
\|\widetilde{X}^{(n)}_{\ion}(t)-X_{\ion}(t)\|^2_{\mathbb{S}^2_T} \leq C\bigg(\EE |\widetilde{\xi}^{(n)}_{\ion}-\xi_{\ion}|^2 +\| G(\ion,\cdot)-G_n(\ion,\cdot) \|_{1}+\int_I \| \tX^{(n)}_{u}-X_{u} \|^2_{\mathbb{S}^2_T} du\bigg).
\end{equation}
Combining \eqref{eq00}, \eqref{eq01} and \eqref{eq:99}, we get
\begin{align*}
\frac{1}{n}\sum_{i=1}^n \| X_{\frac i n}-X^{(n)}_{i}\|^2_{\mathbb{S}^2_T} \leq & C\bigg(\|G-G_n\|_{1}+\frac C n\\
& \quad + \frac{1}{n}\sum_{i=1}^n\EE |\xi^{(n)}_{i}-\xi_{\ion}|^2 +\frac{1}{n}\sum_{i=1}^n\| G(\ion,\cdot)-G_n(\ion,\cdot) \|_{1}
\bigg). 
\end{align*}
If Assumption~\ref{assum:lip} holds, then $G$ is Lipschitz continuous and we have
$$\frac{1}{n}\sum_{i=1}^n\| G(\ion,\cdot)-G_n(\ion,\cdot) \|_{1}\leq \|G-G_n\|_1+\frac C n.$$
Hence we finally obtain \eqref{eq:convg_3}.
Similarly by the result for the maximum difference in Theorem \ref{thm:sta}, we get \eqref{eq:convg_4}. 

We now prove equation \eqref{eq:sss}. We define   $\widetilde{M}^n_i:=\frac{1}{n}\sum_{j=1}^n \zeta^{(n)}_{ij} \delta_{\widetilde{X}^{(n)}_{\jon}}$ \\ and $\widetilde{\Lambda\mu}^n_i:=\frac{1}{n}\sum_{j=1}^nG_n(\ion,\jon)\widetilde{\mu}_{\jon}$, where $\widetilde{\mu}_u:=\cL(\widetilde{X}^{(n)}_u)$. 
We have: 
$$\langle h,\widetilde{ \Lambda\mu}^n_i\rangle-\langle h, \Lambda\widetilde{\mu}(\ion)\rangle=\frac{1}{n}\sum_{j=1}^nG_n(\ion,\jon)\langle h, \widetilde{\mu}_{\jon}\rangle-\int_I G_n(\ion ,v)\langle h,\widetilde{\mu}_v\rangle dv\leq \frac{C}{n} .$$
Moreover, we have
\begin{align*}
& \EE \Bigl[ \langle h, \widetilde{M}^n_i\rangle-\langle h,\widetilde{ \Lambda\mu}^n_i\rangle \Bigr]^2 
 =\frac{1}{n^2}\EE\Bigl[ \sum_{j=1}^n \Bigl(\zeta^{(n)}_{ij}h(\tX^{(n)}_{\jon})-G_n(\ion,\jon)\EE h(\tX^{(n)}_{\jon})\Bigr)\Bigr]^2 \\
&=  \frac{1}{n^2}\sum_{j=1}^n\EE\Bigl[ \zeta^{(n)}_{ij}h(\tX^{(n)}_{\jon})-G_n(\ion,\jon)\EE h(\tX^{(n)}_{\jon})\Bigr]^2  +\frac{1}{n^2}\sum^n_{j=1}\sum_{k\neq j}\EE\Bigl[ \Bigl(\zeta^{(n)}_{ij}h(\tX^{(n)}_{\jon})\\& \qquad \qquad \qquad -G_n(\ion,\jon)\EE h(\tX^{(n)}_{\jon})\Bigr)\Bigl(\zeta^{(n)}_{ik}h(\tX^{(n)}_{\kon})-G_n(\ion,\kon)\EE h(\tX^{(n)}_{\kon}) \Bigr)\Bigr] \leq \frac{\| h\|_{\infty}}{n}, 
\end{align*}
where the last inequality comes from the boundedness of $h$ and the independence of $\zeta^{(n)}_{ij}$. On the other hand, by the previous result, we get
$$\EE\Bigl[ \langle h, M^{(n)}_i\rangle -\langle h, \widetilde{M}^n_i \rangle\Bigr]^2\leq \frac{C}{n}\sum_{j=1}^n\| X^{(n)}_j-\tX^{(n)}_{\jon}\|^2_{\mathbb{S}^2_T}\leq \frac{1}{n}\sum_{i=1}^n\EE |\xi^{(n)}_{i}-\xi_{\ion}|^2+\frac{C}{n}.$$
Combine the above three results, we can conclude that for any $i\in[n]$,
\begin{equation}\label{eq:min}
\EE\Bigl[ \langle h, M^{(n)}_i\rangle -\langle h, \Lambda\widetilde{\mu}(\ion)\rangle\Bigr]^2\leq \frac{1}{n}\sum_{i=1}^n\EE |\xi^{(n)}_{i}-\xi_{\ion}|^2+\frac{C}{n}.
\end{equation}
Finally, by the stability of graphon in Theorem~\ref{thm:sta} and the definition of the operator norm $\|\cdot\|_{\infty\to\infty}$, we get
\begin{align*}
& \langle h, \Lambda\widetilde{\mu}(\ion)\rangle- \langle h, \Lambda\mu(\ion)\rangle \\
\leq & \int_I G_n(\frac{i}{n} ,v)\langle h,\widetilde{\mu}_v\rangle dv-\int_I G(\frac{i}{n},v)\langle h,\mu_v\rangle dv \\
\leq & C\bigg|\int_I \Big( G_n(\ion,v)-G(\ion,v)\Big)\langle h,\widetilde{\mu}_v\rangle dv\bigg|+C\int_IG(\ion,v)(\langle h,\widetilde{\mu}_v\rangle-\langle h,{\mu}_v\rangle)dv\\
\leq &C\|G_n-G\|_{\infty\to\infty}+C\|G_n-G\|_1.
\end{align*}
Combining the above with \eqref{eq:min},
we can conclude that for any $i\in[n]$,
\eqref{eq:sss} holds, 
and
\begin{equation}\label{eq:convg_22}
	\EE\Bigl[ \frac 1 n\sum_{i=1}^n(\langle h, M^{(n)}_i\rangle -\langle h, \Lambda\mu(\ion)\rangle)\Bigr]^2\leq \frac{C}{n}\sum_{j=1}^n \EE| \xi^{(n)}_j-\xi_{\jon}|^2+C \| G_n-G\|_{1}+ \frac{C}{n}.
\end{equation}
Note that if only Assumption~\ref{assum:conti} holds but not 
Assumption~\ref{assum:lip}, we can not get the explicit convergence rate $C/n$ for the last term in the r.h.s. of  \eqref{eq:convg_3}, \eqref{eq:convg_4}, 
\eqref{eq:sss} and \eqref{eq:convg_22}. We just get a sequence of real numbers that converges to $0$ as $n$ goes to $\infty$.

Thus if $\| G_n-G\|_{\infty\to \infty}\to 0$ and $\frac{1}{n}\sum_{j=1}^n \EE| \xi^{(n)}_j-\xi_{\jon}|^2\to 0$, by Markov's inequality, we  obtain that for all $i\in [n]$, as $n\to \infty$, $M^{(n)}_i \to \Lambda\mu(\ion)$ in probability in the weak sense.
By Assumption~\ref{assum:conti} and Corollary \ref{coro:sta}, it is straightforward to verify that in each $I_k$, $v\mapsto\Lambda\mu(v)$ is continuous. By similar arguments as before, it follows that
\begin{align*}
&\EE\Bigl[ \langle h,\frac 1 n\sum_{i=1}^n M^{(n)}_i\rangle -\langle h, \int_I\Lambda\mu(v)dv\rangle\Bigr]^2 \\
\leq & \EE\Bigl[ \frac 1 n\sum_{i=1}^n(\langle h, M^{(n)}_i\rangle -\langle h, \Lambda\mu(\ion)\rangle)\Bigr]^2+\EE\Bigl[ \int_I\big(\langle h,\Lambda\mu(v) \rangle -\langle h, \Lambda\mu(\frac{\lceil nu\rceil}{n})\rangle\big)dv\Bigr]^2,
\end{align*} 
which goes to zero as $n$ goes to $\infty$. This completes the proof of \eqref{eq:dde} in Theorem~\ref{thm:convg_path}.
\end{proof}

\begin{remark}\label{rq:barlambda}
By the law of large numbers, we have that for each $i\in [n]$,\\
{\bluu$|\frac{1}{n}\sum_{j=1}^{n}\zeta^{(n)}_{ij}- \|G_n(\frac{i}{n},\cdot)\|_1|\to 0$}, 
since $G_n$ is a step graphon. Combining this with the proof of Theorem~\ref{thm:convg_path}, we have, under the same assumptions as in Theorem \ref{thm:convg_path}, that all the assertions still hold with $M^{(n)}_i$ and $\Lambda\mu$ replaced by $\overline{M}^{(n)}_i$ and $\bar{\Lambda}\mu$, respectively.   
\end{remark}

\begin{remark}\label{rq:convg}
Note that for any $t\in[0,T]$, $\mathbb{S}^2_T\ni X\mapsto X_t\in \RR$ is {\bluu a contraction}. Thus, for any bounded Lipschitz continuous function $H:\RR\to\RR$, we have that $\bigl|\langle H, M_i^{(n)}(t)\rangle -\langle H,\bar{\Lambda}\mu_t(\ion)\rangle\bigr|\to 0$ {\red with} the same convergence rate as $\bigl|\langle h, M_i^{(n)}\rangle- \langle h,\bar{\Lambda}\mu(\ion)\rangle\bigr|\to 0$, with $h$ being a bounded Lipschitz continuous function from $\cD$ to $\RR$.
\end{remark}

When the graphon is not necessarily continuous, we can still obtain similar convergence results. To this end, we introduce the following definition.\begin{definition}[Continuous modification set]\label{def:contiset}
{\bluu Let $\iota$ be the Lebesgue measure on $I\times I$. For a graphon $G$, we say that a point $(u,v)\in I\times I$ is in the continuous modification set of $G$ if, for any $\epsilon>0$, there exists $A\in\mathcal{B}(I\times I)$ such that $(u,v)\in A$, the restriction of $G$ on $A$ is continuous, and $\iota(A)>1-\epsilon$.
}
\end{definition}

As a corollary of Theorem~\ref{thm:convg_path}, we obtain the following convergence result for general graphons.
{\bluu
\begin{corollary}[Convergence for general graphon]\label{coro:non_conti}
Suppose $\alpha$ is the equilibrium control profile of graphon system \eqref{eq:limitMFBSDE}. Define $\alpha^{(n)}_i=\alpha_{i/n}$ as the control for the $n$-particle system. Let $X^{(n)}$ and $X$ be the solutions of \eqref{eq:finiteMFBSDE} and \eqref{eq:limitMFBSDE}, respectively, with initial conditions $\xi^{(n)}$ and $\xi$, and controls $\alpha^{(n)}$ and $\alpha$. Suppose that for each $n\in\mathbb{N}$, $\{(\ion,\jon), i,j=1,\ldots,n \}$ is in the continuous modification set of $G$. Further assume that  Assumption~\ref{assum:conti} $(i,iii)$ hold, that $\zeta^{(n)}$ satisfies the regularity Assumption~\ref{assum:regu} with $G_n$, and that $\{G_n\}_{n}$ is a sequence of step graphons such that $ \|G-G_n\|_{1}\to 0$. Then if $ \frac{1}{n}\sum_{i=1}^n\EE |\xi^{(n)}_{i}-\xi_{\ion}|^2\to 0$, we have
\begin{equation}\label{eq:90}
	\Ion \sum_{i=1}^n M^{(n)}_i\to \int_I\Lambda\mu(v)dv,
\end{equation}
in probability with $\mu:=\cL(X)$ and
$$\frac{1}{n}\sum_{i=1}^n\| X^{(n)}_i-X_{\frac{i}{n}}\|^2_{\mathbb{S}^2_T}\to 0.$$
Furthermore, if $ \| G_n-G\|_{\infty\to \infty}\to 0$ and $\max_{i\in[n]}\EE |\xi^{(n)}_{i}-\xi_{\ion}|^2\to 0$, then for all $i\in [n]$,
$M^{(n)}_i-\Lambda\mu(\ion)\to 0$, and
$$ \max_{i\in [n]}\| X^{(n)}_{i }-X_{\ion} \|^2_{\mathbb{S}^2_T} \to 0.$$
\end{corollary}
}

\begin{proof}
We only highlight  the  changes from the continuous graphon (Theorem~\ref{thm:convg_path}) to a general one. We  keep the same notation as in the proof of Theorem~\ref{thm:convg_path}. By our assumptions, $(\ion,\jon)_{i,j\in[n]}$ is in the continuous modification set of $G$. By Lusin's Theorem, for any arbitrarily small $\epsilon>0$, we can approximate $G$ by a continuous graphon $\bar{G}_\epsilon$ such that $\bar{G}_\epsilon(\ion,\jon)=G(\ion,\jon)$ and $\|\bar{G}_\epsilon(\ion,\cdot)-G(\ion,\cdot)\|_{1}\leq\epsilon$ for all $i\in[n]$. Let $\bar{X}$ be the controlled dynamics of \eqref{eq:limitMFBSDE} associated with graphon $\bar{G}$, initial condition $\xi$ and control $\bar{\alpha}$. Let $\bar{\mu}_u:=\cL(\bar{X}_u)$. 
Again by Theorem~\ref{thm:sta}, we have
$$\int_I \|\bar{X}_{u}-X_{u}\|^2_{\mathbb{S}^2_T}du\leq \|\bar{G}_\epsilon-G\|_{1}\leq\epsilon.$$
Similarly as obtaining \eqref{eq:99}, we have
\begin{equation}\label{eq:98}
	\|\bar{X}_{\ion}(t)-X_{\ion}(t)\|^2_{\mathbb{S}^2_T} \leq C\bigg(\| G(\ion,\cdot)-\bar{G}_\epsilon(\ion,\cdot) \|_{1}+\int_I \| \tX^{(n)}_{u}-X_{u} \|^2_{\mathbb{S}^2_T} du\bigg).
\end{equation}
Hence we get 
$$\frac{1}{n}\sum_{i=1}^n\|\bar{X}_{\ion}-X_{\ion}\|^2_{\mathbb{S}^2_T}\leq C\epsilon.$$
Finally, by our assumptions, we  obtain that
$ \frac 1 n\sum_{i=1}^n \| X^{(n)}_{i }-X_{\ion} \|^2_{\mathbb{S}^2_T}  \leq C\epsilon+ o(n),$
and similarly
$ \max_{i\in [n]} \| X^{(n)}_{i }-X_{\ion} \|^2_{\mathbb{S}^2_T}  \leq C\epsilon+ o(n).$
Since $\epsilon$ can be chosen arbitrarily small, letting $\epsilon$ go to $0$, we can conclude. By similar arguments as above applied to the difference estimate between $M^{(n)}_i$ and $\Lambda\mu(\ion)$, we  also establish the desired convergence. Finally, since 
\begin{align*}
	&\EE\Bigl[ \langle h,\frac 1 n\sum_{i=1}^n M^{(n)}_i\rangle -\langle h, \int_I\Lambda\mu(v)dv\rangle\Bigr]^2 
	\leq  \EE\Bigl[ \langle h,\frac 1 n\sum_{i=1}^n M^{(n)}_i\rangle -\langle h, \int_I\Lambda\bar{\mu}(v)dv\rangle\Bigr]^2\\ &\qquad \qquad \qquad +\EE\Bigl[ \int_{I\times I}\big(\bar{G}_\epsilon(u,v) \langle h,\bar{\mu}_v\rangle-{G}(u,v) \langle h,{\mu}_v\rangle\big)dvdu\Bigr]^2\\
	\leq & o(n)+C\|\bar{G}_\epsilon-G\|_1+C \int_I \|\bar{X}_{u}-X_{u}\|^2_{\mathbb{S}^2_T}du
	\leq  o(n)+C\epsilon,
\end{align*}
we get 
 \eqref{eq:90}.
\end{proof}

We end this section by providing an example of a graphon that is nowhere continuous yet satisfies the conditions of Corollary~\ref{coro:non_conti}.
\begin{example}[Dirichlet graphon]
Consider the graphon $G$ defined by $G(u,v) = 1$ if $u, v \in \mathbb{Q} \cap [0,1]$, and $G(u,v) = 0$ otherwise. It is clear that $G$ is a measurable function from $I \times I$ to $I$, and hence is a graphon. Although $G$ is nowhere continuous, all rational points belong to the continuous modification set of $G$. Thus, the results of Corollary~\ref{coro:non_conti} apply to this graphon since $\ion$, for $i \in [n]$, are all rational.
\end{example}

\section{Approximate Nash equilibria of finite games}\label{ANE}

{\bluu{
To approximate Nash equilibria for finite games on a network, we  use the connection between large finite systems and the graphon system, particularly the relation between the laws of their state processes.}}
We use the equilibrium control for graphon games as a benchmark to infer the equilibrium for finite games. With the propagation of chaos results, as the population size grows, the distributions of state processes of finite games and graphon games become closer. Intuitively, the equilibrium control for each player in the finite game should be  close to the one taken for the corresponding label in the limit graphon system. It is natural to choose the control associated to label $\frac{i}{n}$ for the $i$-th player in an $n$-player game. When the graphon equilibrium control has some continuity with respect to $u$, we can consider controls associated with labels close to $\ion$.
{\bluu Note that the controls associated with different labels are adapted to different filtrations. For this reason, it is not straightforward to define a control process $\alpha_{u_1}$ for label $u_1 \in I$ from a known $\alpha_{u_2}$ when $u_1 \neq u_2$. To address this, we impose the following assumption, which will holds throughout this section.
\begin{assumption}\label{ass:v}
The jump measures $\nu_v$ are identical for all $v \in I$.
\end{assumption}
Under this assumption, all $(W_u,N_u), u\in I$ have the same distribution of trajectories. 
	For a given $u\in I$ and a control process $\alpha\in \cA^u$, we define $\alpha\hookrightarrow \cA^v$ for $I\ni v\neq u$ as the process $\beta\in\cA^v$ such that ,
  $$
  \beta(t)=a^{\alpha}(t,W_v(\cdot\wedge t),N_v(\cdot\wedge t),\xi_v), \quad t\in [0,T],
  $$
where $a^{\alpha}$ is the corresponding function that generates $\alpha$, as defined in \eqref{eq:controldef}.  

We next define the approximation error for player $i$ as
\begin{equation*}
\epsilon^{(n)}_i(\bm{u}^{(n)}):=\sup_{\beta\in \cA^{\frac{i}{n}}}J_i(\alpha^\star(u^{(n)}_1),\ldots,\alpha^\star(u^{(n)}_{i-1}),\beta,\alpha^\star(u^{(n)}_{i+1})\ldots,\alpha^\star(u^{(n)}_n))-J_i(\bm{\alpha}^{\star}),
\end{equation*}
where $\bm{u}^{(n)}:=(u^{(n)}_1, \cdots, u^{(n)}_n)$ and
$\alpha^\star_i=\alpha^\star(u^{(n)}_i):= \hat{\alpha}_{u^{(n)}_i}\hookrightarrow \cA^{\frac{i}{n}}$, i.e.,
player $i$ adopts the control strategy of the graphon equilibrium control at label $u^{(n)}_i$.

To study the convergence rate, we need an estimate that extends the classical result in \cite{fournier2015rate} on the
convergence rate in Wasserstein distance of the empirical measure of i.i.d. random variables to the graphon framework. Let us denote
\begin{equation}
	\label{eq:def_q}
	\varepsilon_{N,d,\varkappa}\coloneqq \begin{cases}
		N^{-1/2} + N^{-\varkappa/(2+\varkappa)},\; &\text{\rm if } d<4,\;  \text{\rm and } 2+\varkappa\neq 4,\\[0.3em]
		N^{-1/2}\log(1+N) + N^{-\varkappa/(2+\varkappa)},\; &\text{\rm if } d = 4,\; \text{\rm and } 2+\varkappa \neq 4,\\[0.3em]
		N^{-2/d} + N^{-\varkappa/(2+\varkappa)},\; &\text{\rm if } d>4,\; \text{\rm and } 2+\varkappa\neq d/(d-2),	\end{cases}
\end{equation}
where $d$ denotes the dimension of the underlying state space. In this paper, $d=1$, but  the convergence results in this section hold for $d>1$ as well.

We make the following assumption.
\begin{assumption}\label{ass:iniepsilon}
	There exists a constant $\varkappa>0$, such that the family of initial laws  $(\mu_{u,0})_{u\in I}$ satisfies
	$$
	\sup_{u\in I}\int_{\RR^d}|x|^{2+\varkappa}\mu_{u,0}(dx)<\infty.
	$$ 
\end{assumption} 

For $t\in[0,t]$ and $i\in [n]$, we define for $u\in\cI^{(n)}_i$: 
\begin{equation*}
	\overline{M}^{\alpha}_u(t)\coloneqq \frac{1}{\kappa^{(n)}_i}\sum_{j=1}^n\zeta^{(n)}_{ij}\delta_{X^{\alpha}_{\jon}(t)}, \quad \widetilde{M}^{\alpha}_u(t)(dx)\coloneqq \frac{1}{\kappa^{(n)}_i}\sum_{j=1}^n\zeta^{(n)}_{ij}\cL\big(X^{\alpha}_{\jon}(t)\big)(dx), 
\end{equation*}
where $X^\alpha$ is the controlled dynamics of \eqref{eq:limitMFBSDE} under the strategy profile $\alpha\in\cM\cA$. 

\begin{lemma}\label{lem:esempi}
	Suppose that Assumption~\ref{ass:iniepsilon} holds. Then, for any $\alpha\in\cM\cA$, any $t\in[0,T]$ and any $u\in I$,
	$$\EE\big[\cW^2_2\big(\overline{M}^{\alpha}_u(t),\widetilde{M}^{\alpha}_u(t)\big)\big]\leq \varepsilon_{n,d,\varkappa},$$
	where $\varepsilon_{n,d,\varkappa}\to 0$ as $n\to \infty$ is the convergence rate defined in \eqref{eq:def_q}. 
\end{lemma}

\begin{proof}
	Combining Assumption~\ref{ass:iniepsilon} with the definition of $\cM\cA$, it follows from standard estimates for SDEs with jumps, using the Lipschitz property of the coefficients $b$, $\sigma$, and $\ell$, that
	$$
	\sup_{u\in I} \EE\Big[\sup_{0\leq t\leq T} |X^\alpha_u(t)|^{2+\varkappa}\Big]\leq \infty.
	$$
	Therefore, by \cite[Lemma 4.1]{coppini2024nonlinear}, the result follows.
\end{proof}
 
}

%

We are now ready to state the main results of this section. The accuracy and complexity of the approximate equilibria for finite games depend on the underlying graphon and  the way  the network converges to its graphon. 


%
%

The piecewise constant graphon defined in Example~\ref{ex:SBM} is a special case of continuous graphons.

\paragraph{ (Lipschitz) Continuous graphon} 
We call a graphon $G(u,v)$ \emph{continuous} if there exists a collection of intervals $\{I_i, i=1,\ldots ,k\}$, for some $k\in \mathbb{N}$, such that $I=\bigcup_i I_i$, and $G$ is piecewise continuous with respect to $u$ and $v$ in all intervals $I_i, i=1,\ldots,k$. Furthermore, we call it \emph{Lipschitz continuous} if for all $u_1,u_2\in I_i$, $v_1,v_2 \in I_j$, and $i,j\in{1,\ldots,k }$, there exists a constant $C$ such that
$$|G(u_1,v_1)-G(u_2,v_2)|\leq C(|u_1-u_2|+|v_1-v_2|).$$

For each $i \in [n]$, we define $\cI^{(n)}_i := (\partial_-I_j, \ion]$ if $\frac{i-1}{n} \notin I_j$, $\ion \in I_j$; $\cI^{(n)}_i := \left(\frac{i-1}{n}, \frac{i}{n}\right]$ if $\frac{i-1}{n}, \ion \in I_j$; $\cI^{(n)}_i := \left[\frac{i-1}{n}, \partial_+I_j\right)$ if $\ion \in I_j$ and $\frac{i+1}{n} \notin I_j$, where $\partial_-$ and $\partial_+$ denote the lower and upper endpoints of $I_j$, respectively.

\begin{theorem}\label{thm:c}
Suppose that $\zeta^{(n)}$ satisfies the regularity Assumption~\ref{assum:regu} with a step graphon $G_n$ such that $\|G - G_n\|_1 \to 0$.\\
(i)(Continuous graphon).
Suppose Assumption~\ref{assum:conti} holds, $G$ is continuous, and the initial conditions satisfy
$\Ion \sum_{i=1}^n\EE|\xi^{(n)}_i-\xi_{\frac{i}{n}}|^2\to 0.$
Then,
$$\mbox{\rm ess sup}_{\bm{u}^{(n)}\in\cI^{(n)}_1\times\cdots\times \cI^{(n)}_n }\frac{1}{n}\sum_{i=1}^n \epsilon^{(n)}_i(\bm{u}^{(n)})\to 0.$$
Furthermore, if $\|G-G_n\|_{\infty\to \infty}= 0$ and $\max_{i=1,\ldots,n}\EE|\xi^{(n)}_i-\xi_{\frac{i}{n}}|^2\to 0$, then we have 
$$\mbox{\rm ess sup}_{\bm{u}^{(n)}\in\cI^{(n)}_1\times\cdots\times \cI^{(n)}_n }\max_{i=1,\ldots, n} \epsilon^{(n)}_i(\bm{u}^{(n)})\to 0.$$
(ii)(Lipschitz Continuous graphon).
Suppose Assumption~\ref{assum:lip} holds, $G$ is Lipschitz continuous, $\|G-G_n\|_{1} = O(n^{-1})$ and the initial conditions satisfy
$\Ion \sum_{i=1}^n\EE|\xi^{(n)}_i-\xi_{\frac{i}{n}}|^2 =O(n^{-1}).$
Then,
$$\mbox{\rm ess sup}_{\bm{u}^{(n)}\in\cI^{(n)}_1\times\cdots\times \cI^{(n)}_n }\frac{1}{n}\sum_{i=1}^n \epsilon^{(n)}_i(\bm{u}^{(n)}) =O(\varepsilon_{n,d,\varkappa}).$$
Furthermore, if $\sup_{u\in I}\|G(u,\cdot)-G_n(u,\cdot)\|_{1}= O(n^{-1})$ and $\max_{i=1,\ldots,n}\EE|\xi^{(n)}_i-\xi_{\frac{i}{n}}|^2 =O(n^{-1})$, then we have 
$$\mbox{\rm ess sup}_{\bm{u}^{(n)}\in\cI^{(n)}_1\times\cdots\times \cI^{(n)}_n }\max_{i\in[n]} \epsilon^{(n)}_i(\bm{u}^{(n)}) =O(\varepsilon_{n,d,\varkappa}).$$
\end{theorem}

\begin{proof}
{\bluu
The proof relies on the propagation of chaos result from Theorem~\ref{thm:convg_path}. We focus on the Lipschitz continuous case; the continuous case follows similarly by analogous arguments. First, observe that
$$\epsilon^{(n)}_i(\bm{u}^{(n)})\leq \sup_{\alpha\in \cA^{\frac{i}{n}} }\Delta^{(n),1}_i(\alpha,\bm{u}^{(n)})+\sup_{\alpha\in \cA^{\frac{i}{n}} }\Delta^{(n),2}_i(\alpha,\bm{u}^{(n)})+\Delta^{(n),3}_i(\bm{u}^{(n)}),$$
where $\Delta^{(n),1}_i(\alpha,\bm{u}^{(n)}), \Delta^{(n),2}_i(\alpha,\bm{u}^{(n)})$, and $\Delta^{(n),3}_i(\bm{u}^{(n)})$ are defined respectively as 
\begin{align*}
 & \EE\Bigl[ \int_0^T f(t,X^{(n),\alpha,-i}_i(t),\overline{M}^{(n),-i}_i(t),\alpha(t)) dt +g(X^{(n),\alpha,-i}_i(T),\overline{M}^{(n),-i}_i(T))\Bigr] \\
& -\EE\Bigl[ \int_0^T f(t,X^{\star,\alpha,-\ion}_{u^{(n)}_i}(t),\bar{\Lambda}\mu^{\alpha,-\ion}_t(u^{(n)}_i),\alpha(t)) dt +g(X^{\star,\alpha,-\ion}_{u^{(n)}_i}(T),\bar{\Lambda}\mu^{\alpha,-\ion}_T(u^{(n)}_i))\Bigr],\\
& \EE\Bigl[ \int_0^T f(t,X^{\star,\alpha,-\ion}_{u^{(n)}_i}(t),\bar{\Lambda}\mu^{\alpha,-\ion}_t(u^{(n)}_i),\alpha(t)) dt +g(X^{\star,\alpha,-\ion}_{u^{(n)}_i}(T),\bar{\Lambda}\mu^{\alpha,-\ion}_T(u^{(n)}_i))\Bigr] \\
& -  \EE\Bigl[ \int_0^T f(t,X^{\star}_{u^{(n)}_i}(t),\bar{\Lambda}\mu^\star_t(u^{(n)}_i),\alpha^\star_{u^{(n)}_i}(t)) dt +g(X^\star_{u^{(n)}_i}(T),\bar{\Lambda}\mu^\star_T(u^{(n)}_i))\Bigr],
\end{align*}
and
\begin{align*}
& \EE\Bigl[ \int_0^T f(t,X^{\star}_{u^{(n)}_i}(t),\bar{\Lambda}\mu^\star_t(u^{(n)}_i),\alpha^\star_{u^{(n)}_i}(t)) dt +g(X^\star_{u^{(n)}_i}(T),\bar{\Lambda}\mu^\star_T(u^{(n)}_i))\Bigr]\\ 
&- \EE\Bigl[ \int_0^T f(t,X^{(n),\star}_i(t),\overline{M}^{(n),\star}_i(t),\alpha^{(n),\star}_{i}(t)) dt +g(X^{(n),\star}_i(T),\overline{M}^{(n),\star}_i(T))\Bigr].
\end{align*}
Here, $X^{(n),\alpha,-i}$ denotes the state vector of the $n$-player interacting system when the $i$-th player  adopts the control $\alpha\in\cA^{\frac{i}{n}}$, while all other players use the control  $\alpha^{(n),\star}$. The process $X^{\star,\alpha,-\frac{i}{n}}$ is the family of state processes in the limit graphon system where nodes with labels $u \in \cI^{(n)}_i$ follow the control $\alpha\hookrightarrow \cA^u$, and all other labels  $u\in I\setminus \cI^{(n)}_i$ retain the control $\alpha^\star_u$. The measure  $M^{(n),-i}$ denotes the neighborhood empirical measure induced by $X^{(n),\alpha,-i}$, while $\Lambda\mu^{\alpha,-\frac{i}{n}}$ is the graphon mean field generated by   $X^{\star,\alpha,-\frac{i}{n}}$.

We first note that for any choice $\bm{u}^{(n)}\in\cI^{(n)}_1\times\cdots\times \cI^{(n)}$, we can assume that $W_{u^{(n)}_i}, N_{u^{(n)}_i}$, and $W_{\ion}, N_{\ion}$ are the same for each $i\in [n]$, since such a correspondence does not change the law of the graphon systems \eqref{eq:limitMFBSDE} and thus does not affect our approximation. We begin by analyzing $\Delta^{(n),3}_i(\bm{u}^{(n)})$. By the local Lipschitz property of $f$ and $g$ (Assumption~\ref{assm:sta}), we have 
\begin{align}\label{eq:delta3}
	\Delta^{(n),3}_i(\bm{u}^{(n)})\leq & C\EE\big[|X^\star_{u^{(n)}_i}(T)-X^{(n),\star}_i(T)|^2+\cW^2_2(\bar{\Lambda}\mu^\star_T(u^{(n)}_i),\overline{M}^{(n),\star}_i(T))\big] \notag\\
	&+C\bigg( \int_0^T \EE\big[|X^\star_{u^{(n)}_i}(t)-X^{(n),\star}_i(t)|^2+\cW^2_2(\bar{\Lambda}\mu^\star_t(u^{(n)}_i),\overline{M}^{(n),\star}_i(t))\big]dt\bigg).
\end{align}
Using the triangle inequality, we obtain 
\begin{align*}
 \cW^2_2\big(\bar{\Lambda}\mu^\star_t(u^{(n)}_i),\overline{M}^{(n),\star}_i(t)\big)\leq & \cW^2_2\big(\overline{M}^{\star}_{u^{(n)}_i}(t),\overline{M}^{(n),\star}_i(t)\big)+\cW^2_2\big(\overline{M}^{\star}_{u^{(n)}_i}(t),\widetilde{M}^{\star}_{u^{(n)}_i}(t)\big)\\
 &\quad +\cW^2_2\big(\bar{\Lambda}\mu^\star_t(u^{(n)}_i),\widetilde{M}^{\star}_{u^{(n)}_i}(t)\big).
\end{align*}
Applying Lemmas~\ref{lem:G}, \ref{lem:conti}, and \ref{lem:esempi}, along with the Lipschitz assumption of the graphon $G$, we obtain
\begin{align}\label{eq:ssss}
&\EE\big[\cW^2_2\big(\bar{\Lambda}\mu^\star_t(u^{(n)}_i),\overline{M}^{(n),\star}_i(t)\big)\big]\notag\\
& \quad \leq  C\Big(\EE\big[|X^\star_{u^{(n)}_i}(t)-X^{(n),\star}_i(t)|^2\big]+\|G(u^{(n)}_i,\cdot)-G_n(u^{(n)}_i,\cdot)\|_1 
+\varepsilon_{n,d,\varkappa}+\frac C N\Big).
\end{align}
If we choose $u^{(n)}_i=\frac{i}{n}$ for all $i\in[n]$, then we can apply  Theorem~\ref{thm:convg_path} and obtain
$$\frac{1}{n}\sum_{i=1}^n \| X^{(n),\star}_i-X^\star_{\frac{i}{n}}\|^2_{\mathbb{S}^2_T} \leq C\Bigl( \frac{1}{n}\sum_{i=1}^n\EE |\xi^{(n)}_{i}-\xi_{\ion}|^2 +\|G-G_n\|_{1}+ \frac{1}{n} \Bigr),$$
and
$$\max_{i\in [n]} \| X^{(n),\star}_{i }-X^\star_{\ion} \|^2_{\mathbb{S}^2_T}  \leq C\Bigl( \max_{i\in[n]} \EE |\xi^{(n)}_{i}-\xi_{\ion}|^2 +\max_{i\in [n]}\|G(\ion,\cdot)-G_n(\ion,\cdot)\|_{1}+ \frac{1}{n} \Bigr).$$
However, by following similar procedures with $\ion$ replaced by $u^{(n)}_i$, we also obtain
$$\frac{1}{n}\sum_{i=1}^n \| X^{(n),\star}_i-X^\star_{u^{(n)}_i}\|^2_{\mathbb{S}^2_T} \leq C\Bigl( \frac{1}{n}\sum_{i=1}^n\EE |\xi^{(n)}_{i}-\xi_{u^{(n)}_i}|^2 +\|G-G_n\|_{1}+ \frac{1}{n} \Bigr),$$
and
$$\max_{i\in [n]} \| X^{(n),\star}_{i}-X^\star_{u^{(n)}_i} \|^2_{\mathbb{S}^2_T}  \leq C\Bigl( \max_{i\in[n]} \EE |\xi^{(n)}_{i}-\xi_{u^{(n)}_i}|^2 +\max_{i\in [n]}\|G(\ion,\cdot)-G_n(\ion,\cdot)\|_{1}+ \frac{1}{n} \Bigr).$$
Plugging the bounds above and \eqref{eq:ssss} into \eqref{eq:delta3}, and applying the corresponding assumptions for the two types of error, we conclude that
\begin{equation}\label{eq:d3}
\frac 1 n\sum_{i=1}^n\Delta^{(n),3}_i(\bm{u}^{(n)})\leq \varepsilon_{n,d,\varkappa},\quad  \text{and} \quad\max_{i\in [n]}\Delta^{(n),3}_i(\bm{u}^{(n)})\leq \varepsilon_{n,d,\varkappa}.
\end{equation}
 
Now let us consider $\Delta^{(n),1}_i(\alpha,\bm{u}^{(n)})$. Observe that if for all labels $u\in\cI^{(n)}_i$, the control follows the same strategy, then it is straightforward to verify that Lemma~\ref{lem:conti}, Theorem~\ref{thm:sta} and \ref{thm:convg_path} still hold for the control profile $(\alpha^{\star,\alpha,-\ion}_u)_{u\in I}$ defined as: 
\begin{equation*}
	\alpha^{\star,\alpha,-\ion}_u\coloneqq \begin{cases}
		\alpha^\star_u,\; &\text{\rm if } u\notin \cI^{(n)}_i,\\[0.3em]
		\alpha\hookrightarrow \cA^u,\; &\text{\rm if } u\in \cI^{(n)}_i.
	\end{cases}
\end{equation*} 
Then, following the same arguments as in the case of  $\Delta^{(n),3}_i(\bm{u}^{(n)})$, we can obtain that for any $\alpha\in\cA^{\frac{i}{n}}$,
$$\Delta^{(n),1}_i(\alpha,\bm{u}^{(n)})\leq \varepsilon_{n,d,\varkappa}.$$
Hence, we have:
\begin{align}\label{eq:d1}
&	\frac{1}{n}\sum_{i=1}^n\sup_{\alpha\in \cA^{\frac{i}{n}}}\Delta^{(n),1}_i(\alpha,\bm{u}^{(n)})\leq \varepsilon_{n,d,\varkappa},\quad  \text{and} \notag\\
&	\max_{i\in [n]}\sup_{\alpha\in \cA^{\frac{i}{n}}) }\Delta^{(n),1}_i(\alpha,\bm{u}^{(n)})\leq \varepsilon_{n,d,\varkappa}.
\end{align}

Finally we analyze  $\Delta^{(n),2}_i(\alpha,\bm{u}^{(n)})$. By its definition, we have
\begin{align*}
&\Delta^{(n),2}_i(\alpha,\bm{u}^{(n)})=	
\\ & \quad \EE\Bigl[ \int_0^T f(t,X^{\star,\alpha,-\ion}_{u^{(n)}_i}(t),\bar{\Lambda}\mu^{\alpha,-\ion}_t(u^{(n)}_i),\alpha(t)) dt +g(X^{\star,\alpha,-\ion}_{u^{(n)}_i}(T),\bar{\Lambda}\mu^{\alpha,-\ion}_T(u^{(n)}_i))\Bigr] \\
& \qquad  -\EE\Bigl[ \int_0^T f(t,X^{\star,\alpha,-\ion}_{u^{(n)}_i}(t),\bar{\Lambda}\mu^\star_t(u^{(n)}_i),\alpha(t)) dt +g(X^{\star,\alpha,-\ion}_{u^{(n)}_i}(T),\bar{\Lambda}\mu^\star_t(u^{(n)}_i))\Bigr] \\
& \qquad +\EE\Bigl[ \int_0^T f(t,X^{\star,\alpha,-\ion}_{u^{(n)}_i}(t),\bar{\Lambda}\mu^\star_t(u^{(n)}_i),\alpha(t)) dt +g(X^{\star,\alpha,-\ion}_{u^{(n)}_i}(T),\bar{\Lambda}\mu^\star_t(u^{(n)}_i))\Bigr] \\
	& \qquad-  \EE\Bigl[ \int_0^T f(t,X^{\star}_{u^{(n)}_i}(t),\bar{\Lambda}\mu^\star_t(u^{(n)}_i),\alpha^\star_{u^{(n)}_i}(t)) dt +g(X^\star_{u^{(n)}_i}(T),\bar{\Lambda}\mu^\star_T(u^{(n)}_i))\Bigr].
\end{align*}
Following the proof of Theorem~\ref{thm:sta}, one can show that for any $\alpha\in\cA^{\frac{i}{n}}$,
$$\int_I\|X^{\star,\alpha,-\ion}_u-X^\star_u\|^2_{\bS^2_T}du\leq C\int_0^T\int_{I}\EE\big[|\alpha^{\star,\alpha,-\ion}_u(s)-\alpha^\star_u(s)|^2\big]duds\leq \frac C n.$$
Again, by Lemma~\ref{lem:G} and the properties of $f$ and $g$, the difference between the first two terms above can be upper bounded by  $\frac{C}{n}$. Moreover, by the definition of graphon equilibrium, outside a Lebesgue-null set on $\cI^{(n)}_i$, the difference of last two terms is non-positive for any $\alpha\in\cA^{\frac{i}{n}}$. Hence, we conclude that 
$$\mbox{\rm ess sup}_{\bm{u}^{(n)}\in\cI^{(n)}_1\times\cdots\times \cI^{(n)}_n }\max_{i\in[n]}\sup_{\alpha\in \cA^{\frac{i}{n}} }\Delta^{(n),2}_i(\alpha,\bm{u}^{(n)})\leq \frac C n.$$
Combining this with \eqref{eq:d1} and \eqref{eq:d3}, the result follows.
}
\end{proof}

For a non-continuous graphon, we obtain a slightly weaker result.
\begin{proposition}[General graphon]\label{prop:g}
Suppose Assumption~\ref{assum:conti} (i), (iii) hold, and for each $n\in\mathbb{N}$, the set $\{(\ion,\jon) | i,j \in[n] \}$ lies in the continuous modification set of $G$ (see Definition~\ref{def:contiset}). 
Suppose moreover that  $\zeta^{(n)}$ satisfies the regularity Assumption~\ref{assum:regu} with step graphon $G_n$, and $\|G-G_n\|_{1}\to 0$. If the initial condition satisfies
$\Ion \sum_{i=1}^n\EE|\xi^{(n)}_i-\xi_{\frac{i}{n}}|^2\to 0,$
then, as $n\to \infty$, we have
$$\mbox{\rm ess sup}_{\bm{u}^{(n)}\in\cI^{(n)}_1\times\cdots\times \cI^{(n)}_n }\Ion \sum_{i=1}^n \epsilon^{(n)}_i(\bm{u}^{(n)}) \to 0.$$
\end{proposition}
\begin{proof}
We can approximate $G$ by a sequence of continuous graphons. Following similar arguments as in the proof of Theorem \ref{thm:c}, and combining them with the approach used in the proof of Corollary~\ref{coro:non_conti}, we obtain the stated result.
\end{proof}

{\bluu
\paragraph{Sampling graphon}  For any $n\in\mathbb{N}$, let $U_{(1)},\ldots, U_{(n)}$ be independent random variables on $[0,1]$, where for each $i\in [n]$, $U_{(i)}$ is uniformly distributed on the interval $[\frac{i-1}{n},\ion]$. We say that $\zeta^{(n)}$ is \emph{sampled with weights} from the graphon $G$ if $\zeta^{(n)}_{ij} = G(U_{(i)},U_{(j)})$. We say that $\zeta^{(n)}$ is \emph{sampled with probabilities} from the graphon $G$ if $\zeta^{(n)}_{ij}= \text{Bernoulli}(G(U_{(i)},U_{(j)}))$. Note that this sampling procedure differs from the standard approach in the literature (see, e.g., \cite{borgs2019}). In fact, the sampling we consider is piecewise in nature.
}

It is clear that when the interaction strengths $\zeta^{(n)}$ are sampled from a graphon, additional randomness is introduced into the system, which can complicate the analysis. However, as the number of players $n$ becomes large, the effect of this randomness diminishes and does not interfere with the approximation of equilibrium.
\begin{theorem}[Sampling graphon]\label{thm:s}
 Suppose Assumption~\ref{assm:sta} and~\ref{assum:conti} hold. Let $\zeta^{(n)}$ be sampled from a continuous graphon $G$. If the initial conditions satisfy\\
$\Ion \sum_{i=1}^n\EE|\xi^{(n)}_i-\xi_{\frac{i}{n}}|^2\to 0$,
then, for both sampling methods described above, as $n\to \infty$,
$$\mbox{\rm ess sup}_{\bm{u}^{(n)}\in\cI^{(n)}_1\times\cdots\times \cI^{(n)}_n }\frac{1}{n}\sum_{i=1}^n \epsilon^{(n)}_i(\bm{u}^{(n)})\to 0.$$
\end{theorem}
\begin{proof}
	{\bluu
Let $G_n(\bm{U}^{(n)})$ be the step graphon generated from the sampling of the graphon $G$, i.e.,
$$G_n(\bm{U}^{(n)})(u,v)=G\bigl(U^{(n)}_{(\lceil nu\rceil)},U^{(n)}_{(\lceil nv \rceil)}\bigr).$$ 
It is clear that for each realization  $\mathbf{\omega}^{(n)}$ of $\bm{U}^{(n)}$, the matrix $\zeta^{(n)}_{ij}(\mathbf{\omega}^{(n)})$ satisfies the regularity Assumption~\ref{assum:regu} with respect to the graphon $G_n(\mathbf{\omega}^{(n)})$. Then it is straightforward to verify that for each $\mathbf{\omega}^{(n)}$,
$\|G_n(\mathbf{\omega}^{(n)})-G\|_1\rightarrow 0.$
Therefore, for any realization of $\bm{U}^{(n)}$, the conditions in Theorem~\ref{thm:c} (i) are satisfied. As a result, we obtain 
 $$\mbox{\rm ess sup}_{\bm{u}^{(n)}\in\cI^{(n)}_1\times\cdots\times \cI^{(n)}_n }\frac{1}{n}\sum_{i=1}^n \epsilon^{(n)}_i(\bm{u}^{(n)})\to 0.$$
}
\end{proof} 

\section{Conclusion}\label{sec:conclusion}

{\bluu
In this paper, we developed a framework for stochastic graphon games with heterogeneous interactions and jump dynamics. We introduced a controlled graphon mean field system with jumps, allowing control to enter both the drift, diffusion, and jump coefficients. We established well-posedness of the limiting system, studied the limit theory and proved convergence from finite-player games to the limit graphon mean field games. Furthermore, we showed that the graphon equilibrium induces approximate Nash equilibria in the corresponding finite-player games. Our model extends classical mean field games by incorporating heterogeneous network interactions through graphons, as well as control-dependent volatility and jump processes.

This framework opens new directions for modeling and analyzing systemic risk in complex financial networks. A particularly promising application involves modeling the capital dynamics of a continuum of financial institutions indexed by $u \in I = [0,1]$. The controlled dynamics in \eqref{eq:limitMFBSDE} naturally model the evolution of each institution’s capital, influenced by its own risk management action $\alpha_u(s)$, interactions with other institutions via the network structure $G(u,v)$, and exogenous systemic shocks modeled by the jump process. This setting is related to recent models of contagion in financial systems with heterogeneous impact and exposure, see e.g.,~\cite{aminicaosulem2024limit, feinstein2023contagious, nadtochiy2020mean}. However, our framework differs by incorporating graphon structures and continuous-time controlled dynamics with both diffusion and jump components. In our model, the drift and diffusion terms involve aggregation over the graphon $G(u,v)$ and the state distributions $\mu_{v,s}^{\alpha}$, capturing how the capital trajectory of each institution depends on the capital distribution of its counterparties. The jump component models sudden systemic events, such as liquidity freezes or market crashes. Future work may explore numerical schemes and optimal intervention policies in this systemic risk context.

}

\bibliographystyle{plain}
\bibliography{bib_control}

\end{document}